%% file: psrm-sample.tex
\documentclass[
  journal=psrm,
  manuscript=research-note,  
  year=2021,
  volume=2,
]{cup-journal}

\usepackage{booktabs,microtype,siunitx}
\usepackage{amsfonts,amsmath}
\usepackage{amssymb}
\usepackage{multirow}
\usepackage{tikz}
\usepackage{tcolorbox}
\usepackage{wrapfig}
\usepackage{amsthm}
\newtheorem{theorem}{Theorem}
\newtheorem{remark}{Remark}

\newtheorem{lemma}{Lemma}
\newtheorem{assumption}{Assumption}
\newtheorem{definition}{Definition}

\usepackage{amsmath}
\usepackage{amssymb}
\usepackage{mathtools}
\usepackage{amsthm}
\usepackage{amsfonts}       
\usepackage{amsmath}
\usepackage{amssymb}
\usepackage{amsfonts}
\usepackage{nicefrac}       
\usepackage{microtype}      
\usepackage{xcolor}         
\usepackage{graphicx}
\usepackage{enumitem}
\usepackage{multirow}
\usepackage{tikz}
\usepackage{tcolorbox}
\usepackage{wrapfig}
\usepackage{subfigure}
\newtcbox{\alertinline}[1][red]
  {on line, arc = 0pt, outer arc = 0pt,
    colback = #1!20!white, colframe = #1!50!black,
    boxsep = 0pt, left = 1pt, right = 1pt, top = 2pt, bottom = 2pt,
    boxrule = 0pt, bottomrule = 1pt, toprule = 1pt}

\theoremstyle{plain}
 \usepackage{geometry}

\theoremstyle{definition}

\title{Sobolev Acceleration and Statistical Optimality for Learning Elliptic Equations via Gradient Descent}

\author{Yiping Lu}
\affiliation{ICME, Stanford University, CA, USA}
\email{yplu@stanford.edu}

\author{Jose Blanchet}
\affiliation{ICME, Stanford University, CA, USA}
\alsoaffiliation{Department of Management Science \& Engineering, Stanford University, CA, USA}
\email{jose.blanchet@stanford.edu}

\author{Lexing Ying}
\affiliation{ICME, Stanford University, CA, USA}
\alsoaffiliation{Department of Mathematics, Stanford University, CA, USA}
\email{lexing@stanford.edu}


\begin{document}

\input{example-content}


\printbibliography

\appendix

\input{example-appendices}

\end{document}

%% file: example-content.tex
\begin{abstract}
In this paper, we study the statistical limits in terms of Sobolev norms of gradient descent for solving inverse problem from randomly sampled noisy observations using a general class of objective functions. Our class of objective functions includes Sobolev training for kernel regression, Deep Ritz Methods (DRM), and Physics Informed Neural Networks (PINN) for solving elliptic partial differential equations (PDEs) as special cases. We consider a potentially infinite-dimensional parameterization of our model using a suitable Reproducing Kernel Hilbert Space and a continuous parameterization of problem hardness through the definition of kernel integral operators. We prove that gradient descent over this objective function can also achieve statistical optimality and the optimal number
of passes over the data increases with sample size. Based on our theory, we explain an implicit acceleration of using a Sobolev norm as the objective function for training, inferring that the optimal number of epochs of DRM becomes larger than the number of PINN when both the data size and the hardness of tasks increase, although both DRM and PINN can achieve statistical optimality.
\end{abstract}


\section{Introduction}
\label{section:intro}

Several learning based methods for solving inverse problems have been proposed recently with state-of-the-art performance across a wide range of tasks, including medical
image reconstruction \citep{ronneberger2015u}, inverse scattering \citep{khoo2017solving} and 3D reconstruction \citep{sitzmann2020implicit}. In this paper, we study the statistical limit of machine learning methods of solving inverse problems. To be specific, we consider the problem of reconstructing a function from random sampled observations with statistical noise in measurements. We apply gradient descent to a general class of objective functions for the reconstruction. When the observations are the direct observations of the function, the problem is non-parametric function estimation \citep{de2005learning,tsybakov2004optimal}. The observations may also come from  certain physical laws described by a partial differential equation (PDE) \citep{stuart2010inverse,benning2018modern}. Formally, we aim to reconstruct a function $f^*$ based on independently sampled data set $D=\{(x_i,y_i)\}_{i=1}^n$ from an unknown distribution $P$ on $\mathcal{X}\times \mathcal{Y}$, where $y_i$ is the noisy measurement of $u^*$ though a measurement procedure $\mathcal{A}$. For simplicity, we assume $\mathcal{A}$ is self-adjoint in this paper. The conditional mean function $f^*(x)=\mathbb{E}_P(Y|X=x)$ is the ground truth function for observation of $u^*$ through the measurement procedure $\mathcal{A}$, \emph{i.e.} $f^*=\mathcal{A}u^*$. To solve this problem, we consider gradient descending over the following general class of objective function $$
\hat u = \arg\min_{u\in\mathcal{H}} \mathbb{E}_{\mathbb{P}_n(x,y)} \frac{1}{2}\left<u(x),\mathcal{A}_1u(x)\right>-\left<y,\mathcal{A}_2u(x)\right>,
$$
where $\mathbb{P}_n=\frac{1}{n}\sum_{i=1}^n\delta(x_i,y_i)$ is the empirical distribution, $\mathcal{H}$ is a reproducing kernel Hilbert space (RKHS) and $\mathcal{A}_{i}, i= 1,2$ are two self-adjoint operators that satisfy $\mathcal{A}_{1}=\mathcal{A}\mathcal{A}_{2}$. In Section \ref{section:problem}, we show that several algorithms, including kernel regression \citep{de2005learning,caponnetto2007optimal} via Sobolev training \citep{czarnecki2017sobolev,richardson2008sobolev,son2021sobolev} and solving PDEs via machine learning based algorithm, \citep{raissi2019physics,sirignano2018dgm,khoo2017solving,weinan2018deep} can be considered as special cases of this formulation.

Recent works \citep{nickl2020convergence,lu2021machine} have considered the statistical limit of learning of elliptic inverse problem, \emph{i.e.} how many observation of the right hand side function of an elliptic PDE are needed to reach a prescribed performance level. However, none of these papers consider computationally feasible methods for constructing such optimal estimators. In this paper, we consider the statistical optimality of gradient descent \citep{lin2017optimal,pillaud2018statistical,lin2020optimal,marteau2019beyond}, a successful and widely used algorithm in machine learning. We show that proper early stopped gradient descent can achieve information theoretical optimal convergence rate according to a continuous scale of suitable Hilbert norm     (\emph{i.e.} Sobolev norms\citep{fischer2020sobolev,liu2020estimation}, detailed definition see Section \ref{section:problem}). 

We first proof that properly early stopped gradient descent over the class of objective functions can achieve statistical optimality. At the same time, although all the gradient flow of the class of loss function can achieve statistical optimality according to our theory, we discover an acceleration effect of using Sobolev norm as loss function for kernel based machine learning algorithms. The implicit acceleration of Sobolev loss function arises because a differential operator can enlarge the small eigenvalue of kernel integral operator for high frequency functions, leading to better condition numbers and faster convergence in these eigenspaces while keeping the statistical optimality. We justify our theoretical finding with several numerical experiments.

\subsection{Related Works}
\paragraph{Machine Learning Based PDE Solver.} Partial differential equations (PDEs) are widely used in many disciplines of science and engineering and play a prominent role in modeling and forecasting the dynamics of multiphysics and multiscale systems. The recent deep learning breakthrough and the rapid development of sensors, computational power, and data storage in the past decade has drawn attention to numerically solving PDEs via machine learning methods \citep{long2018pde,long2019pde,raissi2019physics,han2018solving,sirignano2018dgm,khoo2017solving}, especially in high dimensions where conventional methods become impractical. Based on the natural idea of representing solutions of PDEs by (deep) neural networks, different loss functions for solving PDEs are proposed. \citep{han2018solving,han2020solvingeigen} utilize the Feynman-Kac formulation which turns solving PDE to a stochastic control problem. The weak adversarial network \citep{zang2020weak} solves the weak formulations of PDEs via an adversarial network. In this paper, we focus on the convergence rate of the Deep Ritz Method (DRM) \citep{weinan2018deep,khoo2017solving} and the Physics-Informed neural network (PINN) \citep{raissi2019physics,sirignano2018dgm}. DRM \citep{weinan2018deep,khoo2017solving} utilizes the variational structure of the PDE, which is similar to the Ritz-Galerkin method in classical numerical analysis of PDEs, and trains a neural network to minimize the variational objective. PINN \citep{raissi2019physics,sirignano2018dgm} trains a neural network directly to minimize the residual of the PDE, i.e., using the strong form of the PDE. Theoretical convergence results for deep learning based PDE solvers has also received considerable attention recently. Specifically,  \citep{lu2021priori,grohs2020deep,marwah2021parametric,wojtowytsch2020some,xu2020finite,shin2020error,bai2021physics} investigated the regularity of PDEs approximated by a neural network and  \citep{lu2021priori,luo2020two,duan2021convergence,jiao2021convergence,jiao2021error} further provided generalization analyses. \citep{nickl2020convergence,lu2021machine,hutter2019minimax,manole2021plugin} provided information theoretical optimal lower and upper bounds for solving PDEs from random samples. However, all these papers assume accessibility of the global solution of empirical loss minimization. In contrast, here we consider the gradient descent algorithm for learning the estimator. The most relevant work in connection to is \citep{nickl2020polynomial}, which considers a polynomial-time Langevin-type algorithms to sample from the posterior measure of the Bayesian inverse methods. Instead of considering the Bayesian setting, here we optimize on the un-regularized objective. However, the estimator is regularized via early stopping \citep{yao2007early,ali2019continuous,ali2020implicit}, \emph{i.e.} we consider the statistical optimality of the implicit regularization effect of optimization algorithm.

 \paragraph{Learning with kernel.} Supervised least square regression in RKHS has a long history and its generalization ability and mini-max optimality has been thoroughly studied \citep{caponnetto2007optimal,smale2007learning,de2005learning,rosasco2010learning,mendelson2010regularization}. Statistical optimality of early stopped (stochastic) gradient descent has been widely discussed in \citep{yao2007early,dieuleveut2016nonparametric,polyak1992acceleration,pillaud2018statistical,lin2017optimal,wei2017early,lei2021generalization}. The convergence of least square regression in Sobolev norm has been discussed recently in \citep{fischer2020sobolev,liu2020estimation}. Recently, training neural networks with stochastic gradient descent in certain regimes has been found to be equivalent to kernel regression \citep{daniely2017sgd,lee2017deep,jacot2018neural}. Gradient descent training of neural network in the kernel regime has been found optimal for non-parametric of a wide class of functions with both early stopping regularization and ridge regression \citep{nitanda2020optimal,hu2020regularization}.

\subsection{Contribution}
\begin{itemize}
\setlength{\itemsep}{0pt}
\setlength{\parsep}{0pt}
\setlength{\parskip}{0pt}
    \item We provide information theoretical lower bounds (Theorem \ref{thoerem:lowerbound}) for a wide class of inverse problems, including the Sobolev learning rate \citep{fischer2020sobolev} for the solution of elliptic inverse problems. We also show that the previous lower bound \citep{nickl2020convergence,lu2021machine} for machine learning solving elliptic equations can be considered as a special case of our lower bound. 
    
    \item We provide a proof of statistical optimality of the gradient descent algorithm of a general class of objective functions (Theorem \ref{theorem:earlystopping}), including PINN \citep{raissi2019physics,sirignano2018dgm} and Deep Ritz Methods \citep{weinan2018deep,khoo2017solving} for solving PDEs as well as Sobolev training \citep{son2021sobolev,czarnecki2017sobolev,yu2021gradient} of kernel methods.  We provide \citep{lu2021machine} a computational feasible estimator and generalize the previous statistical optimality results of gradient descent \citep{yao2007early,pillaud2018statistical,lin2020optimal} to general Sobolev norm.
    \item We also characterize the acceleration effect of Sobolev loss function for learning with kernel. The acceleration happens because differential operator can enlarge the small eigenvalues for high frequency functions, leading to better condition number and faster convergence in these eigenspaces while keeping the statistical optimality. Thus when the target function have more high frequency component, the lead of PINN will become larger (Figure \ref{fig:Online}). We justify our theoretical finding with several numerical experiments (Figure \ref{fig:fitfunction} and Figure \ref{fig:staticscheq}).

\end{itemize}

\section{Problem Formulation}
\label{section:problem}

In this section, we formulate the problem of learning inverse problem using the kernelized gradient descent. As described previously, we aim to reconstruct a function $f^\ast\in\mathbb{R}^\mathcal{X}$ from random observations of $u^\ast=\mathcal{A}f^*$, where $\mathcal{A}$ is an observation process which is modeled by an operator maps from $\mathbb{R}^\mathcal{X}$ to $\mathbb{R}^\mathcal{X}$. To solve this problem, we write the operator $\mathcal{A}$ in terms of two operators $\mathcal{A}_i \text{ } (i=1,2)$ with $\mathcal{A}_1=\mathcal{A} \mathcal{A}_2$ and build our objective function as
\begin{equation}
    \begin{aligned}
       \mathbb{E}_{\mathbb{P}} \left[ \frac{1}{2}\left<u(x),\mathcal{A}_1u(x)\right>-\left<y,\mathcal{A}_2u(x)\right>\right],
    \end{aligned}
    \label{eq:objective}
\end{equation}
where $\mathbb{P}$ is the joint distribution of $x$ and $y$ with $x$ sampled from the uniform distribution on $\mathcal{X}$ for simplicity and $y$ as the noisy observation of $f(x)=(\mathcal{A}u)(x)$. In other words, $\mathbb{E}(y|x)=f(x)$. The minimizer of objective function (\ref{eq:objective}) is the ground truth function $u^\ast=\mathcal{A}^{-1}f$ that we are interested in.

\paragraph{Learning with Kernel} Consider the case that $u$ is parameterized by a Reproducing Kernel Hilbert Space $u_\theta(x)=\left<\theta,K_x\right>$ (we provide standard notations of RKHS in Appendix \ref{Appendix:notations}). At the same time, the kernel function has the following representation $K(s,t)=\sum_{i=1}^\infty \lambda_i e_i(s)e_j(t)$, where $e_i$ are orthogonal basis of $\mathcal{L}_2(\rho_\mathcal{X})$ with $\rho_\mathcal{X}$ being the uniform distribution over $\mathcal{X}$. Then $e_i$ is also the eigenvector of the covariance operator $\Sigma = \mathbb{E}_{x\sim\mathbb{P}}K_x\otimes K_x$ with eigenvalue $\lambda_i>0$, \emph{i.e.} $\Sigma e_i = \lambda_i e_i$. Here $g\otimes h=gh^\top$ is an operator from $\mathcal{H}$ to $\mathcal{H}$ defined as
\[
    g\otimes h:f\rightarrow \left<f,h\right>_{\mathcal{H}}g.
\]
The covariance matrix $\Sigma$ is the core of the integral operator technique \citep{smale2007learning,caponnetto2007optimal} for kernel regression. For any $f\in\mathcal{H}$, the reproducing property gives
\[
\small
(\Sigma f)(z)
=\left<K_z,\Sigma f\right>_{\mathcal{H}}=\mathbb{E}[f(X)k(X,z)]=\mathbb{E}[f(X)K_x(X)].
\]
If we consider the mapping $S:\mathcal{H}\rightarrow L_2(dx)$ defined as a parameterization of a vast class of functions in $\mathbb{R}^{\mathcal{X}}$ via $\mathcal{H}$ through the mapping $(Sg)(x)=\left<g,K_x\right>$ ($\Phi(x)=K_x=K(\cdot,x)$). Its adjoint operator $S^\ast:\mathcal{L}_2\rightarrow\mathcal{H}$ then can be defined as $g\rightarrow \int_\mathcal{X}g(x)K_x \rho_X(dx)$. $\Sigma$ is the same as the self-adjoint operator $S^\ast S$ and the self-adjoint operator $\mathcal{L}=SS^\ast:L_2\rightarrow L_2$ can be defined as
\[
    (\mathcal{L}f)(x) =(SS^*f)(x)=\int_\mathcal{X} K(x,z)f(z)\rho_{\mathcal{X}}(dz).
\]
Based on this notation, we present all our assumptions on the underlying kernel.

\begin{assumption}[Assumptions on Kernel]
\setlength{\itemsep}{0pt}
\setlength{\parsep}{0pt}
\setlength{\parskip}{0pt}
\label{assumption:kernel} We assume the standard capacity condition on kernel covariance operator with a source condition about the regularity of the target function following \citep{caponnetto2007optimal}. We further assume a regularity condition for our kernel $k(\cdot,\cdot)$ via a $\ell_\infty$ embedding property follows \citep{steinwart2009optimal,dicker2017kernel,pillaud2018statistical,fischer2020sobolev}. These conditions are stated explicitly below.
\begin{itemize}
\setlength{\itemsep}{0pt}
\setlength{\parsep}{0pt}
\setlength{\parskip}{0pt}
\item \textbf{(a) Standard assumptions.} 
The kernel feature are bounded almost surely, \emph{i.e.} $|k(x,y)|\le R$ and the observation $y$ is also bounded by $M$ almost surely.
\item \textbf{(b) Capacity condition.} Consider the spectral representation of the kernel covariance operator $\sigma=\sum \lambda_i e_i\otimes e_i$, we assume polynomial decay of eigenvalues of the covariance matrix $\lambda_i\propto i^{-\alpha}$ for some $\alpha>1$. As a result $Q=\text{tr}(\Sigma^{1/\alpha})<\infty$. 
\item \textbf{(c) Source condition.} We also impose an assumption on the smoothness of the true function. There exists $\beta\in(0,1]$ such that $u^\ast=\mathcal{L}^{\beta/2}\phi$ for some $\phi\in L^2$. If $u^\ast(x)=\left<\theta_\ast,K_x\right>_{\mathcal{H}}$, the source condition can also be written as
    $$
||\Sigma^{\frac{1-\beta}{2}}\theta_\ast||_{\mathcal{H}}<\infty.
$$

\item \textbf{(d) Capacity conditions on $\mathcal{A}_i$.} For theoretical simplicity, we assume that the self-adjoint operators  $\mathcal{A}_i$ are diagonalizable in the same orthonormal basis $e_i$. Thus we can assume
$$
\mathcal{A}_1 = \sum_{i=1}^\infty p_ie_i\otimes e_i,\mathcal{A}_2 = \sum_{i=1}^\infty q_ie_i\otimes e_i
$$
for positive constants $p_i,q_i>0$. We further assume $p_i\propto i^{-p}$ and $q_i\propto i^{-q}$. This commuting assumptions also made in \citep{cabannes2021overcoming,de2021convergence}. due to the Bochner's theorem. We further assume $p<0,q<0,\alpha+p>0$. We refer the detailed discussion to Remark \ref{remark:assumption}.

\item \textbf{(e) Regularity results on RKHS.} For $\mu\in[0,1]$, there exists $\kappa_\mu\ge 0$ such that $\Phi(x)\otimes\Phi(x)\le k_\mu^2R^{2\mu}\Sigma^{1-\mu}$ holds almost surely.  The regularity assumption here is equivalent to $||g||_{L_\infty}^2\le\kappa_\mu^2 R^{2\mu}||\Sigma^{1/2-\mu/2}g||_{\mathcal{H}}^2$  and implies $||g||_{L_\infty}\le \kappa_\mu R^\mu||g||_{\mathcal{H}}^\mu||g||_{L_2}^{1-\mu}$ for every $g\in\mathcal{H}$. As a consequence, we know that $||\Sigma^{\mu/2-1/2}\Phi(x)||_{\mathcal{H}}\le\kappa^\mu R^{\mu}$ holds almost surely. \citep{steinwart2009optimal,fischer2020sobolev,pillaud2018statistical} 
\end{itemize}
\end{assumption}

\begin{remark}\label{remark:assumption}
To simplify the technical exposition, we  assume that operator $\mathcal{A}_i (i=1,2)$ commute with the kernel covariance operator $\Sigma$.  This assumption is  also made in \citep{de2021convergence,cabannes2021overcoming}. Here we provide several examples that satisfy this assumption. The simplest case is $\mathcal{A}_1=\mathcal{A}_2=id$ , which gives rise to the function regression setting. For numerically solving a PDE, we take $\mathcal{A}_i$ to become the power of the Laplace operator $\Delta$. If the domain is a sphere, the eigen-functions are spherical harmonics which are also the eigen-functions of a wide class of kernels, examples includes the dot product kernels \citep{scetbon2021spectral} and the Neural Tangent Kernel \citep{bietti2020deep,chen2020deep}, when the data distribution is uniform distribution. When the domain is the torus, the eigen-functions are Fourier modes. If we consider a shift invariant kernel $K(x,y)=\psi(x-y)$, from Bochner’s Theorem $K(x,y)=\sum_{i=1}^n\tilde{\psi}(w)e^{iws}e^{-iwt}$ we know that the eigen-functions are also Fourier modes. There are also works that use Green function as the kernel \citep{zhou2011semi,fasshauer2011reproducing}, where the three operators will automatically commute with each other.
\end{remark}

In this paper, we consider the convergence of the estimator in Sobolev norm class. We define the different Sobolev spaces via the power space approaches used in \citep{steinwart2012mercer,fischer2020sobolev}. 
\begin{definition}[Sobolev Norm] For $\gamma>0$, the \emph{$\gamma$-power space} is
\[
\mathcal{H}^\gamma:=\left\{\sum_{i\ge 1}a_i\lambda_i^{\gamma/2}e_i:\sum_{i\ge 1}a_i^2\le\infty\right\}\subset L_2(v),
\]
equipped with the $\gamma$-power norm via $||\sum_{i\ge 1}a_i\lambda_i^{\gamma/2}e_i||_\gamma:=\left(\sum_{i\ge 1}a_i^2\right)^{1/2}$.

\end{definition}

It is obvious that $||\mathcal{L}^{\gamma/2}f||_\gamma=||f||_{L_2}$ and $||f||_\gamma\le||\Sigma^{\frac{1-\gamma}{2}}f||_{\mathcal{H}}$ \citep{fischer2020sobolev}. The source condition can also be understood as the target function $u^\ast$ lies in the $\beta$-power Sobolev space. The regularity condition of the kernel function implies a continuously  embedding from $\mathcal{H}^\gamma\rightarrow L_\infty$. Throughout this paper, we consider the convergence rate of $\hat u-u^\ast$ in $\gamma$-power Sobolev norm ($\gamma>0$).

\subsection{Examples}

\paragraph{Sobolev Training} \citep{shi2010hermite,czarnecki2017sobolev,son2021sobolev}  introduce the idea of training using Sobolev spaces via matching not only the function value but also the derivative of the classifier. Using different Sobolev norms as loss function has also been used widely in image processing, inverse problems, and graphics applications \citep{yang2021implicit,calder2010image,richardson2008sobolev,yu2021repulsive,yu2021repulsiveb,soliman2021constrained}. The work of \citep{yang2021implicit} discovered that different Sobolev loss functions would lead to different implicit bias and that the proper Sobolev preconditioned gradient descent can accelerate the optimization of geometry objectives \citep{yu2021repulsive,yu2021repulsiveb,soliman2021constrained}. In this paper, we discover that stochastic gradient descent over Sobolev norm loss class functions can achieve statistical optimal but proper selection of the Sobolev norm loss function can accelerate training. We call this phenomenon \textbf{Sobolev Implicit Acceleration} and discuss it in Section \ref{section:Sobolev}.

\paragraph{Machine Learning Based PDE Solver.}  To simplify the exposition, we focus on a prototype elliptic PDE: Poisson's equation on a torus, \emph{i.e.} $\Omega=\mathbb{T}^d=[0,1]^d_{\text{per}}$. Our focus is on the analysis of deep-learning-based numerical methods for the elliptic equations 
\begin{equation}\label{eq:maineq}
\begin{aligned}
   -\Delta u +u & = f \quad \text{ in } \Omega.
\end{aligned}
\end{equation}
We mainly focus on analyzing Deep Ritz Method (DRM) \citep{weinan2018deep} and Physics Informed Neural Network (PINN) \citep{raissi2019physics,sirignano2018dgm}.  DRM solves the equation (\ref{eq:maineq}) via minimizing the following variational form
\begin{equation}\label{eq:variationalform}
    u^\ast = \arg \min_{u\in\mathcal{F}} \mathcal{E}^{\text{DRM}}(u):=\frac{1}{2} \int_{\Omega} |\nabla u|^2+u^2 \ dx- \int_{\Omega} f u dx, 
\end{equation}
while PINNs solves the equation (\ref{eq:pinnform}) via minimizing the following strong formula, \emph{i.e} the residual of the PDE,
\begin{equation}\label{eq:pinnform}
    u^\ast = \arg \min_{u\in\mathcal{F}} \mathcal{E}^{\text{PINN}}(u):=\frac{1}{2} \int_{\Omega} \left(\Delta u-u+f\right)^2 \ dx, 
\end{equation}
where $u$ is minimized over a parameterized function class $\mathcal{F}$ (for example neural network). Here we consider the function class to be the RKHS space \citep{chen2021solving,stepaniants2021learning}. \citep{lu2021machine} showed that empirical risk minimization of both objectives can achieve information theoretical optimal bounds. The objective function in \ref{eq:variationalform} and \ref{eq:pinnform} can be considered as special case of objective function (\ref{eq:objective}). For DRM, $\mathcal{A}_1 u=\Delta u$ and $\mathcal{A}_2u =u$ for all function $u\in\mathbb{R}^{\mathcal{X}}$. For PINN, $\mathcal{A}_1 u=\Delta^2 u$ and $\mathcal{A}_2u =\Delta u$ for all function $u\in\mathbb{R}^{\mathcal{X}}$. 

We discover that PINN convergences faster than DRM consistently due to the implicit Sobolev acceleration, matching the observation made in \citep{chen2020comprehensive}. \citep{cabannes2021overcoming} considered semi-supervised learning using Laplacian regularization with kernel parameterization.  However, this paper does not consider training with stochastic gradient descent and also does not introduce the source condition assumption that leads to different convergence rate for a hierarchical parameterization of task difficulty.

\section{Main Theorem}

We present our main results in this section, including an information theoretical lower bound and a matching upper bound with proper selected early stopping time.

\subsection{Lower Bounds}
\label{section:lowerbound}

This subsection investigates the statistical optimality of the Sobolev convergence rate of solving elliptic problem using stochastic gradient descent. We provide the information theoretical lower bound of learning the elliptic problems. Different from \citep{nickl2020convergence,lu2021machine}, we formulate the problem in an RKHS. This leads to a different construction of hypothesis and show that \citep{nickl2020convergence,lu2021machine} is a special case of our lower bound using specific kernel and operator $\mathcal{A}_i (i=1,2)$ in Section \ref{section:discuss}.

\begin{theorem}[Lower Bound]\label{thoerem:lowerbound}
Let $(X,B)$ be a measurable space,  $H$ be a separable RKHS on $X$ with respect to a bounded and measurable kernel $k$ and operator $\mathcal{A}=(\mathcal{A}_2^{-1}\mathcal{A}_1)$ satisfies Assumption \ref{assumption:kernel}. We have $n$ i.i.d. random observations $\{(x_i,y_i)\in\mathcal{X}\times\mathcal{Y}\}_{i=1}^n$ of $f^\ast = \mathcal{A}u, u\in\mathcal{H}^\gamma\cap L_\infty$, \emph{i.e.} $y_i=f^*(x_i)+\eta_i$ where $\eta_i$ is a mean zero random noise satisfies the momentum assumption $\mathbb{E}|\eta|^m\le\frac{1}{2}m!\sigma^2L^{m-2}$ for some constants $\sigma,L>0$. Then for all estimators $H:(\mathcal{X}\times\mathcal{Y})^{\otimes n}\rightarrow \mathcal{H}^\gamma$ satisfies
\[
\inf_{H}\sup_{u^\ast}\mathbb{E}||H(\{(x_i,y_i)\}_{i=1}^n)-u^\ast||_\gamma^2\gtrsim n^{-\frac{(\max\{\beta,\mu\}-\gamma)\alpha}{\max\{\beta,\mu\}\alpha+2(q-p)+1}}.
\]
\end{theorem}

\subsection{Upper Bounds}
\label{section:upper}
This subsection, we consider the (multiple pass)  gradient descent over the empirical data of objective function (\ref{eq:objective}).  We aim to construct our estimator via optimizing the empirical loss function
\[
\sum_{i=1}^n  \frac{1}{2}u(x_i)\mathcal{A}_1u(x_i)-y_i\mathcal{A}_2u(x_i),
\]
where $x_i$ is sampled randomly and $y_i$ is the associated noisy observation introduced in Section \ref{section:intro}. We consider a parameterization $u(x)=\left<u,K_x\right>$ and $\mathcal{A}_iu(x)=\left<\mathcal{A}_i\theta,K_x\right>_\mathcal{H}=\left<\theta,\mathcal{A}_i K_x\right>_\mathcal{H}$ and express our empirical objective function as
\begin{equation}
    \begin{aligned}
    &\mathbb{E}_{\mathbb{P}_n(x,y)} \frac{1}{2}\left<u(x),\mathcal{A}_1u(x)\right>-\left<y,\mathcal{A}_2u(x)\right>
    \\&= \mathbb{E}_{\mathbb{P}_n(x,y)}\frac{1}{2} \left<u,K_x\right>\left<\mathcal{A}_1u,K_x\right>-y\left<\mathcal{A}_2u,K_x\right>\\
    &=\mathbb{E}_{\mathbb{P}_n(x,y)} \frac{1}{2} \left<u,K_x\otimes \mathcal{A}_1 K_x u\right>-y\left<u,\mathcal{A}_2K_x\right>
    \end{aligned}
\end{equation}
Then the gradient descent algorithm can be written as the following procedure:
\begin{itemize}
\setlength{\itemsep}{0pt}
\setlength{\parsep}{0pt}
\setlength{\parskip}{0pt}
    \item \textbf{Initialization}: $\theta_0=\bar\theta_0=0$, $\gamma$ is a constant to be determined later which is used as the learning rate in the algorithm.
    \item \textbf{Iteration}: For the $t-$th iteration, we perform the following gradient descent step
    $$
    \theta_t=\theta_{t-1}+\gamma\frac{1}{n}\sum_{i=1}^n\left(y_{i}\mathcal{A}_2 K_{x_{i}}-\left<\theta_{t-1},\mathcal{A}_1K_{x_{i}}\right>_{\mathcal{H}}K_{x_{i}}\right)
    $$
    with an averaging step $\bar\theta_t=(1-\frac{1}{t})\bar\theta_{t-1}+\frac{1}{t}\theta_t$.
\end{itemize}
\paragraph{Remark.} Note that the optimizing dynamics considered here is not the exacting gradient descent dynamics over the empirical objective. The gradient of the quadratic term $\frac{1}{n}\sum_{i=1}^nu(x_i)\mathcal{A}_1u(x_i)$  should be $\frac{1}{n}\sum_{i=1}^n\left(
\left<\theta_{t-1},\mathcal{A}_1K_{x_{i}}\right>_{\mathcal{H}}K_{x_{i}}+\left<\theta_{t-1},K_{x_{i}}\right>_{\mathcal{H}}\mathcal{A}_1K_{x_{i}}\right)$ but we take instead $\frac{1}{n}\sum_{i=1}^n\left<\theta_{t-1},\mathcal{A}_1K_{x_{i}}\right>_{\mathcal{H}}K_{x_{i}}$  in our dynamics. In the population expectation, the two dynamics are the same due to the commuting assumption between the kernel integral operator and operator $\mathcal{A}_1$. The implementation of our dynamics can be applied via a stop gradient operator on the $\mathcal{A}_1 u$ during back propagating. The slight variation of gradient descent considered here facilitates the technical analysis.

The following theorem is the main result for upper bounds with the proof details given in the appendix.
\begin{theorem}\label{theorem:earlystopping} Under Assumption \ref{assumption:kernel}, we have the following three regimes shown in Figure \ref{fig:phasediag}.
\begin{itemize}
\setlength{\itemsep}{0pt}
\setlength{\parsep}{0pt}
\setlength{\parskip}{0pt}
\item For $\beta>\frac{\alpha+2q-p-1}{\alpha}$, if we take $t=n$ and \alertinline{$\gamma = n^{\frac{\alpha+p}{\beta\alpha+2(p-q)+1}-1}$}, we obtain the following rate
    $$
    \mathbb{E}[||\bar{\theta_t}-u^\ast||_\gamma^2]=O(n^{-\frac{(\beta-\gamma)\alpha}{\alpha\beta+2(p-q)+1}}).
    $$
    \item For $\frac{\alpha+2q-p-1}{\alpha}\le\beta\le\frac{\mu\alpha+2q-p+1}{\alpha}$, if we take \alertinline{$t = n^{\frac{\alpha+p}{\beta\alpha+2(p-q)+1}}$} and $\gamma$ a small enough constant, we obtain the following rate
    $$
    \mathbb{E}[||\bar{\theta_t}-u^\ast||_\gamma^2]=O(n^{-\frac{(\beta-\gamma)\alpha}{\alpha\beta+2(p-q)+1}}).
    $$
    \item For $\beta>\frac{\mu\alpha+2q-p+1}{\alpha}$, if we take \alertinline{$t = n^{\frac{\alpha+p}{\mu\alpha+p}}$} and $\gamma$ a small enough constant, we obtain the following rate
    $$
    \mathbb{E}[||\bar{\theta_t}-u^\ast||_\gamma^2]=O(n^{-\frac{(\beta-\gamma)\alpha}{\mu\alpha+p}}),
    $$
    which is not an optimal converging rate.
\end{itemize}
\end{theorem}

\paragraph{Sketch of the Proof.} We first rewrite the averaged gradient descent in a more compact formula as
$\eta_0=0, \eta_u=\eta_{u-1}+\gamma (\mathcal{A}_2^\top\hat S_n^\ast \hat y-\hat \Sigma_{Id,\mathcal{A}_1}\eta_{t-1})$ where $\hat S_n:\mathcal{H}\rightarrow \mathbb{R}^n$ is defined as $\hat S_n g=\frac{1}{\sqrt{n}}\left(g(x_1),\cdots,g(x_n)\right)$, $\hat\Sigma_{\mathcal{O}_1,\mathcal{O}_2}=\frac{1}{n}\sum_{i=1}^n \mathcal{O}_1K_x\otimes \mathcal{O}_2K_x$ and $Id$ is the identity operator.   For the error of GD, we consider early stopping of gradient descent algorithm as a spectral filtering \citep{gerfo2008spectral,pillaud2018statistical,blanchard2018optimal,lin2020optimal}. Our proof is based on standard bias-variance decomposition. For $t$ iteration, GD will behave similarly to ridge regression with $\gamma t$ regularization strength \citep{yao2007early,pillaud2018statistical} and this result in bias of $(\frac{1}{\gamma t})^{\frac{(\beta-\gamma)\alpha}{\alpha+p}}$. For the variance, we provide a bound which is related to the effective dimension given by $ \text{tr}((\Sigma_{Id,\mathcal{A}_1}+(\frac{1}{\gamma t})I)^{-1}\Sigma_{\mathcal{A}_2^\top\mathcal{A}_2})$ and obtain a final variance of the form $\frac{1}{n}(\gamma t)^{-\frac{\gamma\alpha+p}{\alpha+p}}{(\frac{1}{\gamma t})^{-\frac{1}{\alpha+p}}(\frac{1}{\gamma t})^{-\frac{p-2q}{\alpha+p}}}+\frac{1}{n}(\frac{1}{\gamma t})^{-\frac{\gamma\alpha+p}{\alpha+p}}(\frac{1}{\gamma t})^{-\frac{\mu\alpha-p}{\alpha+p}}(\frac{1}{\gamma t})^{\frac{\beta\alpha-2q}{\alpha+p}}$. If we only have the first term of variance, we shall achieve information theoretical optimal bound when \alertinline{$t = n^{\frac{\alpha+p}{\beta\alpha+2(p-q)+1}}$}. For the section term in the variance is from the convergence of empirical covariance matrix $\hat \Sigma_{Id,\mathcal{A}_1}$ to the population one $\Sigma_{Id,\mathcal{A}_1}$. This term can be reduced using semi-supervised learning techniques as in \citep{murata2021gradient,lu2021machine}.

{

\begin{figure}
    \centering

\begin{tikzpicture}
\centering
\draw[style=help lines,step=0.5cm,opacity=0.3] (0,0) grid (5.3,3.3);

\draw[->,very thick] (0,0) -- (5.8,0) node[right] {$\alpha$};
\draw[->,very thick] (0,0) -- (0,3.8) node[above] {$\beta$};

\draw[line width =4pt] (0,0) to [in=180,out=90] (5,2);

\draw[line width =4pt] (5,0)to [in=-40,out=180](0.3,1.0);

\filldraw[fill= orange!50,opacity=0.3] (0.3,1) to [in=180,out=40] (5,2)--(5,0)to [in=-40,out=180](0.3,1.0);

\filldraw[fill= black!50,opacity=0.3] (5,0)to [in=-40,out=180](0.3,1.0) to [in=90,out=40] (0.0,0.0)--(5,0);

\node[below,scale=0.7] at (4.2,2){$\beta=\frac{\alpha+2q-p-1}{\alpha}$};

\node[below,scale=0.7] at (4.2,0){$\beta=\frac{\mu\alpha+2q-p+1}{\alpha}$};
\node[above,orange,scale=0.8] at (4.2,0.2){Constant LR};

\node[below,black,scale=0.8] at (1.2,0.0){Sub-Optimal};

\filldraw[fill= blue!50,opacity=0.3] (0,0) to [in=180,out=90] (5,2)--(5,3)--(0,3)--(0,0);
\node[above,violet,scale=0.8] at (4.2,3){Small LR, $n$ Iteration};


\end{tikzpicture}
    \caption{Phase diagram of different regimes for solving inverse problem using stochastic gradient descent. Except the gray regime, GD can achieve information theoretical rate.}
    \label{fig:phasediag}
\end{figure}
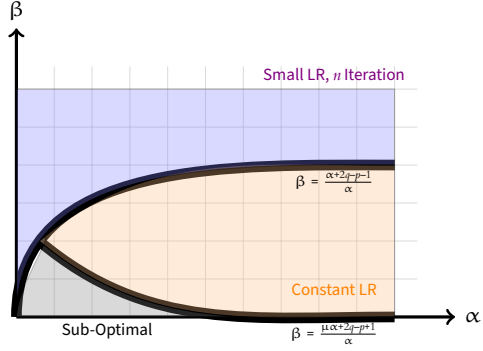

}

\subsection{Discussion and Implication of Our Theory}

\label{section:discuss}

\paragraph{Relationship with \citep{nickl2020convergence,lu2021machine}.} \citep{nickl2020convergence,lu2021machine} provided a lower bound of the form $n^{-\frac{{\color{orange}2\alpha-2s}}{2\alpha-\textbf{\color{red}4t}+d}}$ for a $2t-$th order linear PDE $\Delta^t u = f$ with solution in $H^\alpha$, evaluated in $H^s$ norm. We shall discuss the relationship between their bound with our $n^{-\frac{{\color{orange}(\beta-\gamma)}\alpha}{\beta\alpha+2\textbf{\color{red}(p-q)}+1}}$ lower bound based on the kernel representation of Sobolev spaces. The numerator $(\beta-\gamma)$ matches the $\alpha-s$ term in \citep{nickl2020convergence,lu2021machine}'s lower bound and the $q-p$ term is the order of the linear PDE which matches the $t$ term in the denominator in \citep{nickl2020convergence,lu2021machine}'s lower bound. The spectral decay speed of kernel $\alpha$ is always relative to the dimension $d$. To understand this problem, we consider the following two examples.

For the first example, the kernel is defined on the torus $\mathbb{T}^d=[0,1]^d_{\text{per}}$. We consider the  space of square integrable functions on $\mathbb{T}^d$ with mean 0 and the Mat\'{e}rn kernel $K_{\sigma,l,v}(x,y)=\sigma^2\frac{2^{1-v}}{\Gamma(v)}\left(\frac{|x-y|}{l}\right)^vB_v(\frac{|x-y|}{l})$, where $B_v$ is the modified Bessel function of section kind. The covariance operator is $C_\theta=\sigma^2(-\Delta+\tau^2 I)^{-s}$ with orthonormal eigenfunctions $\phi_m(x)=e^{2\pi i\left<m,x\right>}$ and corresponding eigenvalues $\lambda_m=\sigma^2(4\pi^2|m|^2+\tau^2)^{-s}$ for every $m\in\mathbb{Z}^d\backslash\{0\}$ \citep{stein1999interpolation}.

For the second example, we consider the Mercer's decomposition of a translation invariant kernel via Fourier series
$K(s-t)=\frac{1}{2\pi}\sum_w \tilde{K}(w)e^{iw(s)}e^{iw(-t)}dw$. The eigenfunctions of the translation invariant kernel is the Fourier modes and the eigenvalues are the Fourier coefficients. As an example, for Neural Tangent Kernel, \citep{cao2019generalization,chen2020deep,bietti2020deep,nitanda2020optimal} proved that the corresponding $\alpha=\frac{d}{d-1}$ and the eigenfunctions are spherical harmonics that diagonalize the differential equation.

For the upper bound, \citep{lu2021machine} established the convergence rate based on the \emph{empirical process} technique \citep{mendelson2010regularization,steinwart2009optimal}, while our paper switches to the integral operator/inverse problem technique \citep{de2005learning,smale2007learning,caponnetto2007optimal}. An advantage of the integral operator/inverse problem technique is that it can provide convergence results with respect to a continuous scale of Sobolev norms while the empirical process technique can only be used for the Sobolev norm equivalent to the objective function.

\paragraph{Relationship with \citep{shi2010hermite}} \citep{shi2010hermite} also considered learning from data involving function value and gradients under the framework of least-square regularized regression in reproducing kernel Hilbert spaces. In this paper, we only have access to the noisy observation of the function values but still aim to know about the convergence rate with respect to the Sobolev norm. At the same time, we further consider an inverse problem setting with an early stopping regularization, which is not discussed in \citep{shi2010hermite}. However, we introduce a commuting assumption over the differential operator with the kernel integral operator that makes the problem easier.

\paragraph{Sobolev Implicit Acceleration} Below we discuss the implication of the choice of early stopping time $t = n^{\frac{\alpha+p}{\beta\alpha+2(p-q)+1}}$. First of all, the best early stopping time here does not depend on $\gamma$, which means the best model in different Sobolev is the same over the stochastic gradient descent path asymptotically. Secondly, all the components in an iteration step depend on the problem itself except the numerator $\alpha+p$. For differential operators, the $p$ is actually \textbf{\emph{negative}} (differential operators have large eigenvalues over high-frequency basis). Thus we can accelerate the training  via letting $p$ more negative, \emph{i.e.} using a higher order Sobolev norm as loss can lead to earlier stopping. As an implication, the PINN achieves the statistical optimal solution faster than DRM.


\paragraph{Relationship with implicit bias of frequency} Recent work credit the success of deep learning to the fast training in low frequency components \citep{rahaman2019spectral,xu2019frequency,kalimeris2019sgd}. However, in our work, with Sobolev preconditioning, the training speed of high frequency part increases, yet achieving statistical optimality in the class of Sobolev norm. This suggests that the implicit bias of frequency is not necessary for good generalization results. We also would like refer to \citep{amari2020does,mucke2020stochastic} Theorem 8 for the extreme case, where the authors directly invert the population covariance matrix which leads to the same training speed in every eigen-spaces while still maintaining the statistical optimality in $\ell_2$ norm. However the preconditioning matrix in \citep{amari2020does} is the population Fisher information matrix, which requires further sampling of unlabeled data that is not accessible in our setting.

\paragraph{Discussion of the Sub-Optimal Regime} In the sub-optimal regime, the concentration error between the empirical covariance matrix $\hat \Sigma_{Id,\mathcal{A}_1}$ and the population one $ \Sigma_{Id,\mathcal{A}_1}$ dominates. With the observation that these concentrations have no relationship with the supervision signal, \citep{lu2021machine,murata2021gradient} proposed to utilize the semi-supervised learning to reduce the error in this regime. In \citep{lu2021machine}, Deep Ritz method requires semi-supervised learning while PINN does not for the exact empirical risk minimization solution. In our formulation, if $|p|$ is larger, the sub-optimal regime will become smaller, which contradict with the observation in \citep{lu2021machine}. However \citep{lu2021machine} only considers the statistical generalization bound but doesn't take optimization into consideration. We leave designing algorithm with smaller sub-optimal regime as future work.

\section{Sobolev implicit acceleration}
\label{section:Sobolev}

The Sobolev norm has already been proposed as loss function for training neural network \citep{czarnecki2017sobolev} and solving PDEs \citep{son2021sobolev,yu2021gradient}. However, all these papers need a further gradient information of the supervision signal. This does not fit the theoretical framework  considered here and hence it is also not fair to compare their algorithms with methods without gradient supervision signal. Thus in this section, we proposed an alternative objective that can perform Sobolev training without gradient supervision loss function. The basic idea is to using an integration by parts
\begin{equation}
    \begin{aligned}
    \int |\nabla u-\nabla f|^2 dx = \int ||\nabla u||_2^2 +2\Delta u\cdot f+||\nabla f||_2^2 dx,
    \end{aligned}
    \label{eq:objsobolev}
\end{equation}
which leads to an objective function without the gradient of the target function. In this section, we shall show how this idea is applied to different machine learning examples.

\subsection{Predicting a Toy Function on Torus}

In this section, we conduct experiments to illustrate the Sobolev implicit acceleration for function regression. Different from the Sobolev training \citep{czarnecki2017sobolev}, the objective that we are interested in does not involve the gradient of the target function. As a result, we do not need to train a teacher network to provide the gradient supervision information as done in \citep{czarnecki2017sobolev}. In the toy example, for simplicity we ignore the boundary terms introduced by the integral by part. Here consider estimating a function on the torus, \emph{i.e.} a periodic function. We consider using $\int \lambda||u-f||^2+||\nabla u||_2^2 +2\Delta u\cdot f+||\nabla f||_2^2 dx$ as our objective function. The goal is to fit function  $y=\sum_{i=1}^d\sin(2\pi x_i)$ using Gaussian Kernel and a simple three layer feed-forward network with tanh activation function. We randomly sampled 1000 data in 10 dimension as our dataset and run a gradient descent algorithm. Figure \ref{fig:fitfunction} presents our convergence result of the validation error, where the Sobolev norm have shown an acceleration effect for training.
 
\begin{figure}
    \centering
    \includegraphics[width=2.5in]{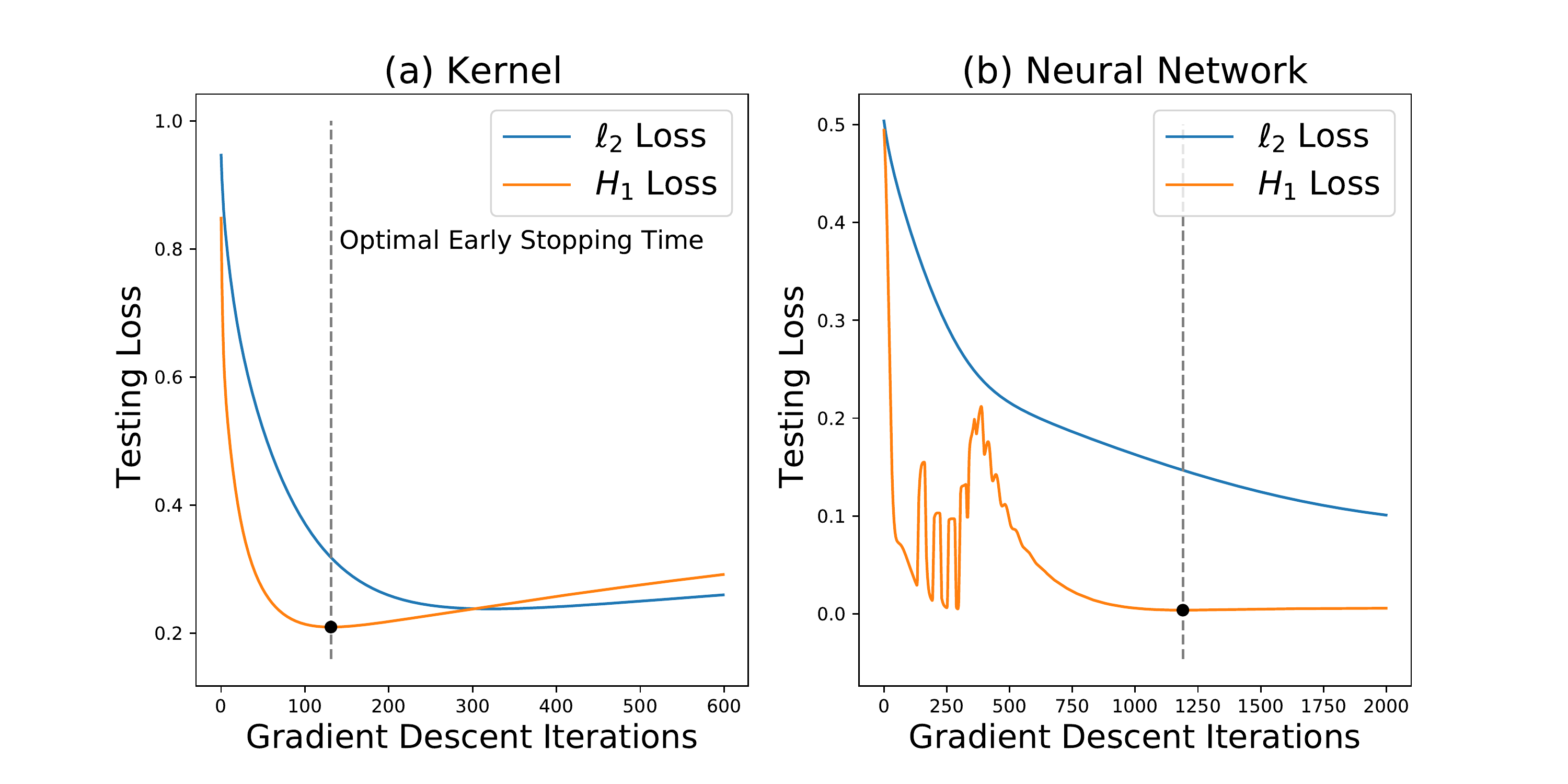}
    \caption{Sobolev Implicit Acceleration of Estimating function using kernel method and Neural Network. We observed that using Sobolev Norm as loss function can accelerate training.}
    \label{fig:fitfunction}
\end{figure}

\subsection{Solving Partial Differential Equations}

In this section, we conduct experiments to illustrate the Sobolev implicit acceleration for solving partial differential equation using PINNs \citep{raissi2019physics,yu2021gradient} in 3 dimensions. The example is a simple Poisson equation (static schr\"{o}dinger equation) on the torus
\begin{equation}
    \begin{aligned}
    \Delta u +u = f \text{ in } \mathcal{T}^d=[0,1]^d_{\text{per}}.
    \end{aligned}
    \label{eq:poisson}
\end{equation}

We first compare the Physics Informed Neural Network \citep{raissi2019physics} and Deep Ritz Method \citep{weinan2018deep,khoo2017solving}  with online random inputs. To enforce the periodic boundary conditions, we add a penalty term $\mathcal{L}_{b}=\int_{(x,y)\in[0,1]^2} (u(x,y,0)-u(x,y,1))^2+(u(0,x,y)-u(1,x,y))^2+(u(x,0,y)-u(x,1,y))^2dxdy$ to match the periodic condition of the function value and 
another term $\mathcal{L}_{b,grad}=\int_{(x,y)\in[0,1]^2} (\nabla u(x,y,0)-\nabla u(x,y,1))^2+(\nabla u(0,x,y)- \nabla u(1,x,y))^2+(\nabla u(x,0,y)-\nabla u(x,1,y))^2dxdy$ to match the periodic condition of the function value. We tested PINN and Deep Ritz on both $u(x)=\sum_{i=1}^d\sin(2x_i)$ and $u(x)=\sum_{i=1}^d\sin(4x_i)$. We use the same experiment setting as \citep{chen2020comprehensive} and keep the learning rate constantly to $1e-3$ to match our theory. 50000 data points are randomly sampled in every batch. The results are shown in Figure \ref{fig:Online}. PINN converges  faster than DRM consistently in terms of iteration number and the lead seems to become significant for more oscillatory problems.

\begin{figure}[h]
	\centering
	\subfigure[Smooth Problem]{
			\includegraphics[width=0.2\textwidth]{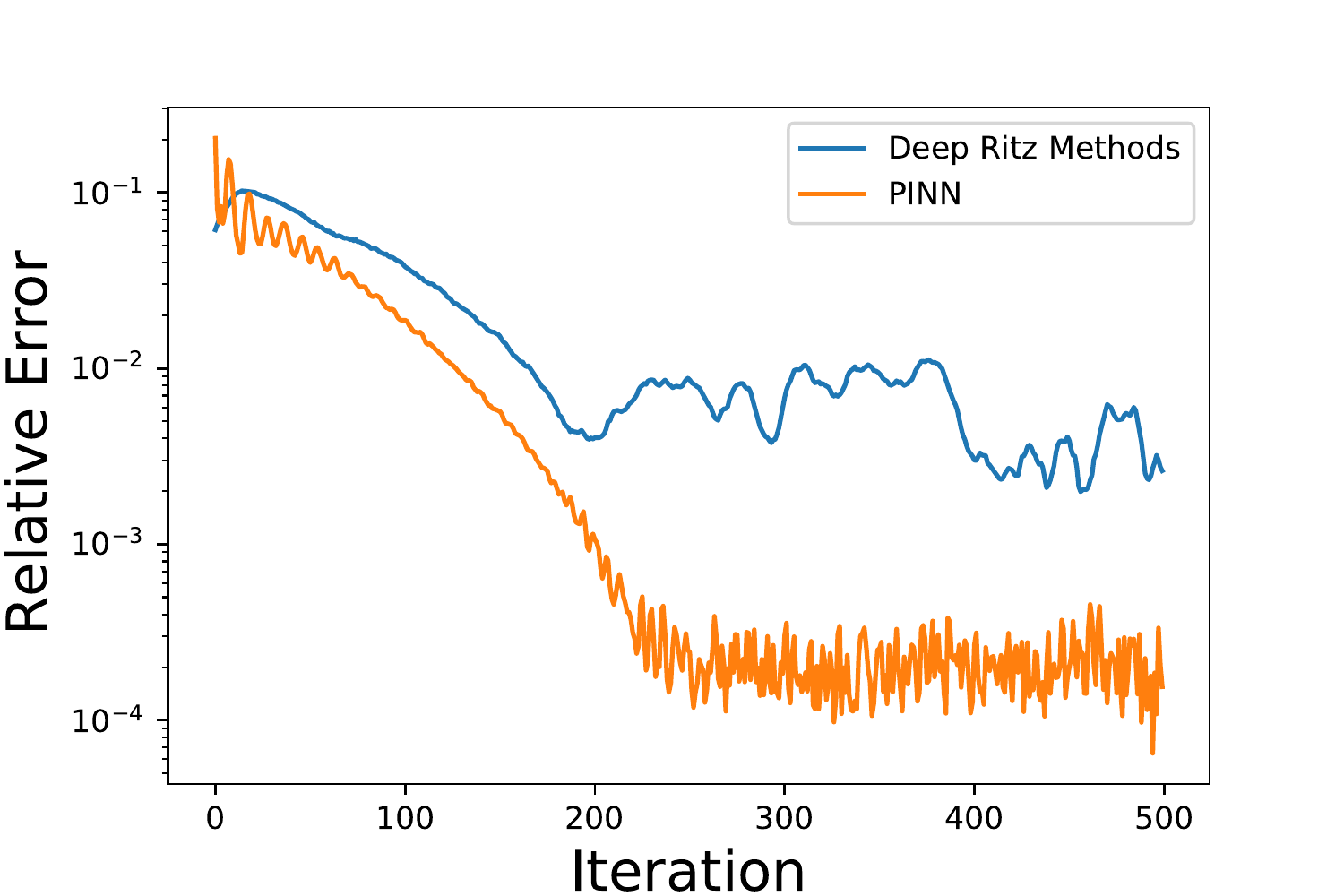}
		\label{fig:cifaravg}
	}
	\subfigure[Harder Problem]{
			\includegraphics[width=0.2\textwidth]{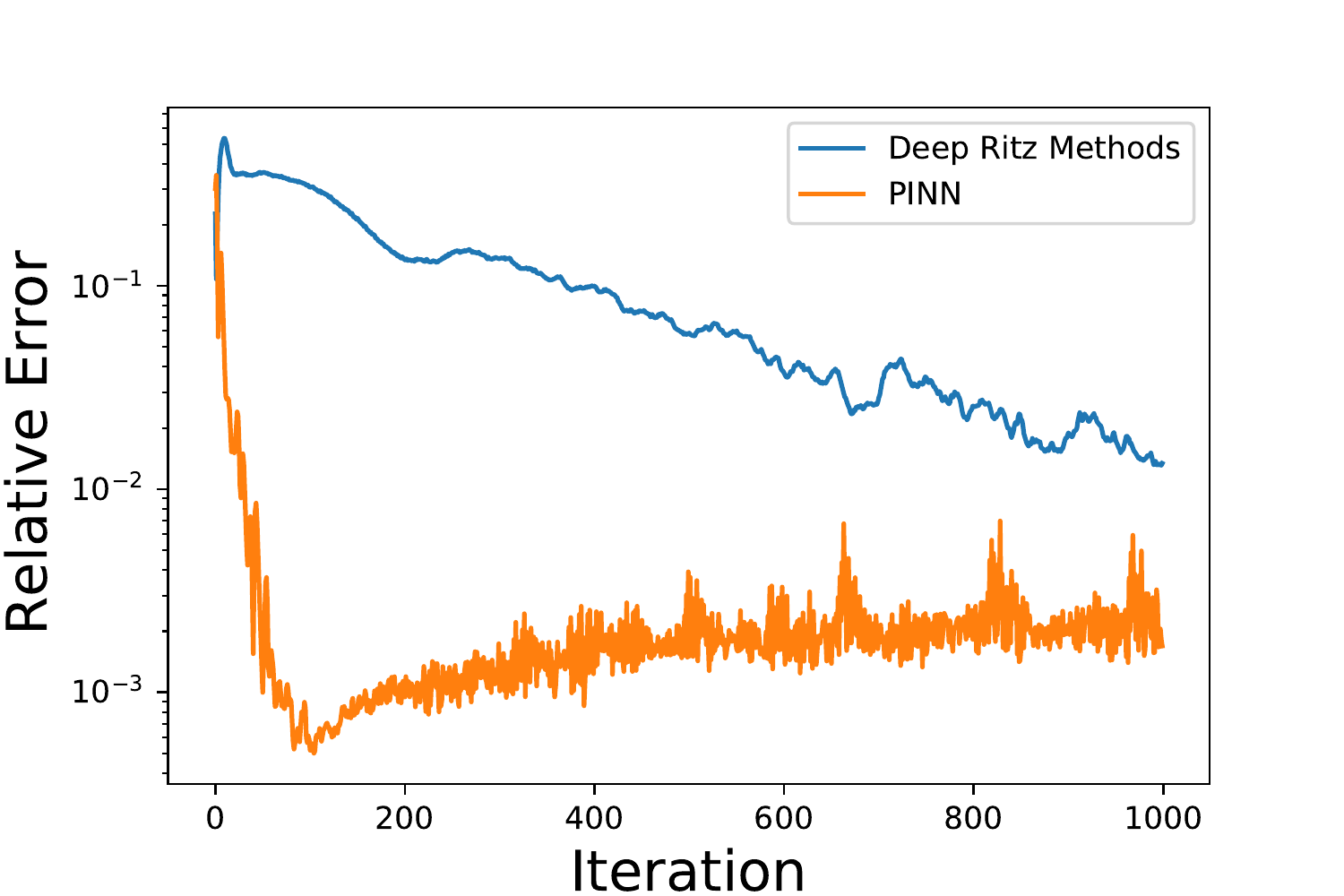}
		\label{fig:cifarworst}
	}
	
	\caption{We show the convergence result of PINN and Deep Ritz Method for smooth problem $\sum_{i=1}^d\sin(2\pi x)$ and harder problem $\sum_{i=1}^d\sin(4\pi x)$. PINN convergence faster than DRM for online stream input which also matches our theory and the empirical observation in \citep{chen2020comprehensive}. The Sobolev Implicit Acceleration will becomes more significant for harder problem as our theory shows.}
	\label{fig:Online}

\end{figure}

To solve equation (\ref{eq:poisson}), we consider minimizing the following Sobolev norm objective function
\[
\mathcal{L}(u):=\lambda||\Delta u+u - f||_{L_2(\Omega)}^2+ ||\nabla\Delta u+\nabla u - \nabla f||_{L_2(\Omega)}^2.
\]
\citep{son2021sobolev,yu2021gradient} also considered using Sobolev norms as the loss function. \citep{son2021sobolev} showed that the Sobolev norms exhibit an acceleration effect. However, in our setting, we cannot have random samples of $\nabla f$. To avoid information of $\nabla f$ appearing in the objective function, we perform an integration by parts that leads to the following objective function
\begin{equation*}\small 
    \begin{aligned}
    &\mathcal{L}_{grad}=\int ||\nabla \Delta u(x)+\nabla u(x)-\nabla f(x)||_{2}^2dx\\ &=\int ||\nabla \Delta u(x)||_2^2+||\nabla u(x)||_2^2+||\nabla f(x)||_2^2\\
    &+2\nabla\Delta u(x)\nabla u(x)-2\nabla u(x)\nabla f(x)-2\nabla\Delta u(x)\cdot \nabla f(x)dx\\
     &=\int ||\nabla \Delta u(x)||_2^2+||\nabla u(x)||_2^2+||\nabla f(x)||_2^2\\
    &+2\nabla\Delta u(x)\nabla u(x)+2\Delta u(x) f(x)+2\Delta\Delta u(x)\cdot  f(x)dx.
    \end{aligned}
\end{equation*}

We conduct the Sobolev training with the objective function $\mathcal{L}_{pinn}+\lambda \mathcal{L}_{grad}+\lambda_1 \mathcal{L}_{b}+\lambda \mathcal{L}_{b,grad}$ and compare it with PINN and DRM. Following mostly the experiment setting in \citep{chen2020comprehensive}, we fix 3000 random samples as the dataset and run stochastic gradient descent with batchsize 50. The result presented in Figure \ref{fig:staticscheq}
show the Sobolev implicit acceleration, i.e., the gradient dynamic of higher order Sobolev norm convergence faster. We do not scale the Sobolev training to online setting as under large batch size the Sobolev training consume too much memory at this point.

\begin{figure}[htbp]
\centering
\subfigure[Solution by Sobolev Training.]{
\centering
\begin{minipage}[t]{0.5\linewidth}
\centering
\includegraphics[width=1.5in]{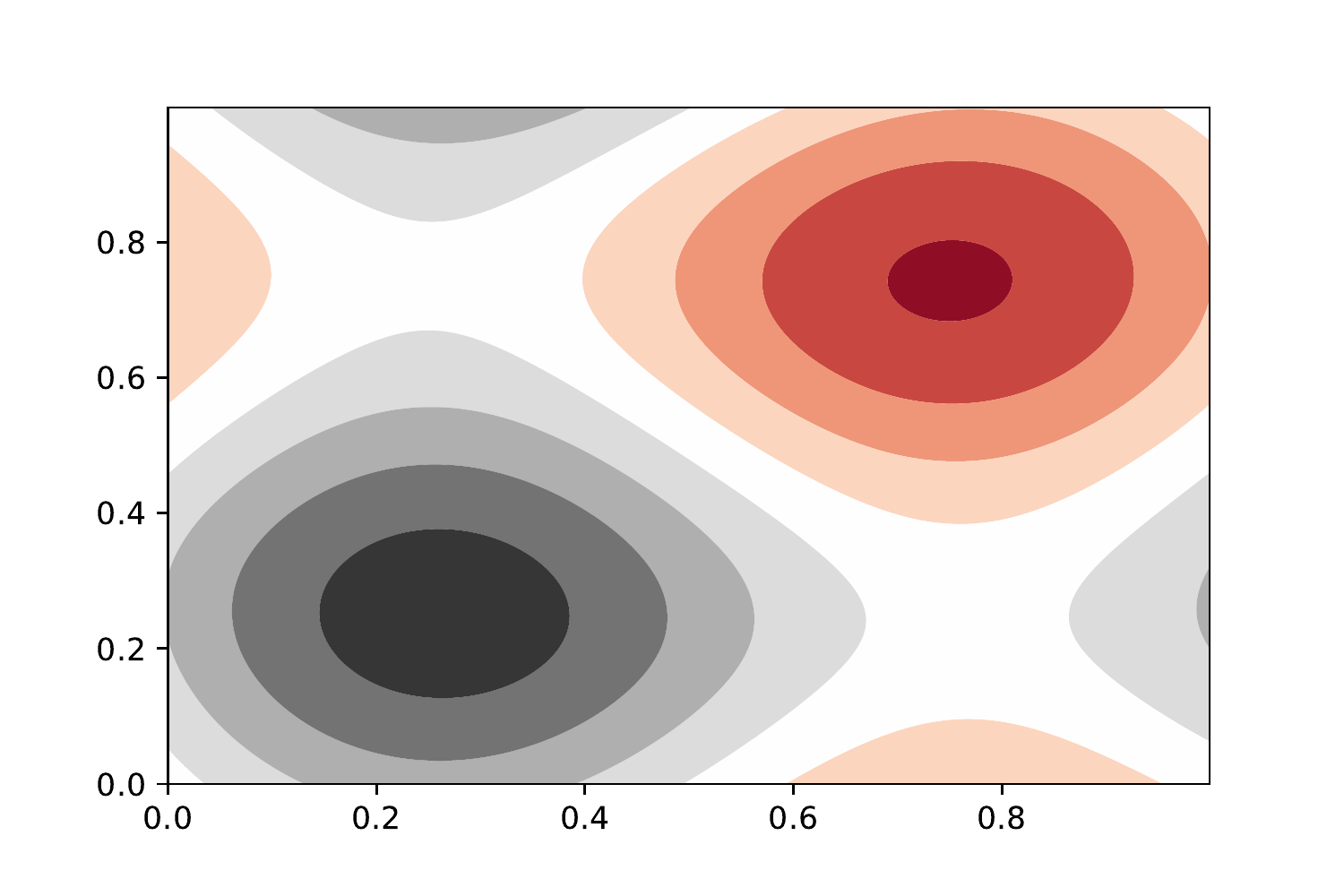}
\end{minipage}%
}%
\subfigure[Convergence Speed.]{
\centering
\begin{minipage}[t]{0.5\linewidth}
\centering
\includegraphics[width=1.5in]{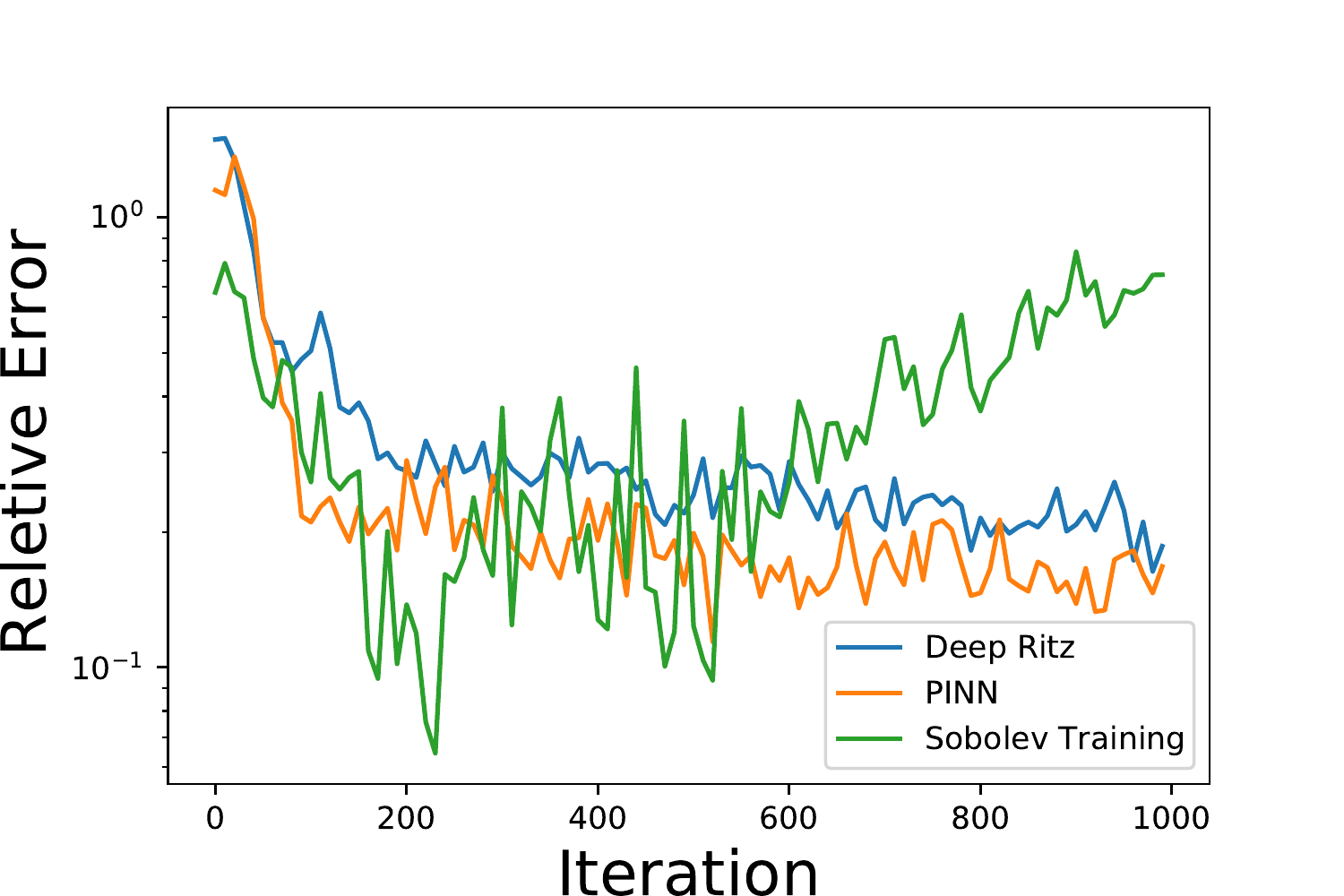}
\end{minipage}%
}%

\centering
\caption{Solving equation (\ref{eq:poisson}) in 3 dimension with 3000 fixed samples using Deep Ritz Method \citep{weinan2018deep}, Physics-Informed Neural Network \citep{raissi2019physics} and Sobolev Training. }
\label{fig:staticscheq}

\end{figure}

\section{Conclusion and Discussion}

In this paper, we consider the statistical optimality of gradient descent for solving elliptic inverse problem using a general class of objective functions. Although we can achieve statistical optimality of gradient descent using all the objective functions with proper early stopping time, the early stopping iteration strategy for the optimal solution behaves differently as a function of the sample size. For instance, we observed that PINN convergences faster than the DRM method. Generally speaking, by using a higher order Sobolev norm as loss function, one can accelerate training. The reason is that  the differential operator can counteract the kernel integral operator, leading to better condition number for optimization. We call this phenomena {\em Sobolev implicit acceleration}.

Although we have shown the Sobolev implicit acceleration on several simple examples, the $\Delta^s u$ term is hard to compute in high dimensions, scalable Sobolev training without gradient supervision in higher dimension remains as future work. However, we believe that this direction is promising. For example, we can use MIM method \citep{lyu2020mim,zhu2021local} to accelerate the training.  It is also interesting to generalize our results beyond GD, for example to mirror descent \citep{vavskevivcius2020statistical} and accelerated gradient descent \citep{pagliana2019implicit}.  In this paper, we did not consider operators with continuous spectrum and it will be interesting to extend our results using the techniques in \citep{colbrook2021computing}. Due to technical issue, we have not considered the batch stochastic gradient descent. It will be  interesting to characterize the condition under which the stochastic noise in gradient does not degrade the optimal bounds that we obtain. At the same time, we also want to investigate more complex nonlinear inverse problems as \citep{abraham2019statistical,monard2021consistent} considered. It is also interesting to consider inverse problem arising from integral equation where $p>0$.

\begin{acknowledgement}
 Yiping Lu is supported by the Stanford Interdisciplinary Graduate
Fellowship (SIGF). Jose Blanchet is supported in part by the Air Force Office of Scientific Research under award number FA9550-20-1-0397 and NSF grants 1915967, 1820942, 1838576. Lexing Ying is supported by National Science Foundation under award DMS-2011699.  Yiping Lu also thanks Yifan Chen, Junbin Huang, Zong Shang, Bin Dong and George Em Karniadakis for helpful comments and feedback.
\end{acknowledgement}

%% file: example-appendices.tex
\newpage

{\Huge \textbf{Appendix.}}

The appendix is constructed as follows:
\begin{itemize}
\setlength{\itemsep}{0pt}
    \item In Appendix \ref{Appendix:notations}, we introduce the basic notations of Reproducing Kernel Hilbert space and the associated kernel integral operator. We also put a discussion of how differential operators and Sobolev spaces relates to the kernel setting we considered as a preliminary.
    \item In Appendix \ref{appendix:gd}, we consider the statistical optimality of the early stopped gradient descent algorithm. We bound the difference of the gradient descent.
    \item In Appendix \ref{appendix:low}, we provide our proof for the lower bound in Section \ref{section:lowerbound} using the Fano method.
\end{itemize}

\section{Preliminaries and Notations}
\label{Appendix:notations}

This section starts with an overview of reproducing kernel Hilbert space, including Mercer's decomposition, the integral operator techniques \citep{smale2007learning,de2005learning,caponnetto2007optimal,fischer2020sobolev} and the relationship between RKHS and the Sobolev space \citep{adams2003sobolev}. In order to fit the objective function we considered, we did a slight modification to the original integral operator technique \citep{smale2007learning,de2005learning,caponnetto2007optimal}.

\subsection{Reproducing Kernel Hilbert Space}

We consider a Hilbert space $\mathcal{H}$ with inner product $\left<,\right>_{\mathcal{H}}$ is a separable Hilbert space of functions $\mathcal{H}\subset \mathbb{R}^{\mathcal{X}} $. We call this space a Reproducing Kernel Hilbert space if $f(x)=\left<f,K_x\right>_{\mathcal{H}}$ for all $K_x\in\mathcal{H}:t\rightarrow K(x,t),x \in \mathcal{X}$. Now we consider a distribution $\rho$ on $\mathcal{X}\times \mathcal{Y} (\mathcal{Y}\subset \mathbb{R})$ and denote $\rho_X$ as the margin distribution of $\rho$ on $\mathcal{X}$. We further assume $\mathbb{E}[K(x,x)]<\infty$ and $\mathbb{E}[Y^2]<\infty$. We define $g\otimes h=gh^\top$ is an operator from $\mathcal{H}$ to $\mathcal{H}$ defined as
$$
g\otimes h:f\rightarrow \left<f,h\right>_{\mathcal{H}}g.
$$
At the same time, we knows that
$$
||f\otimes g||=||f||_\mathcal{H}||g||_\mathcal{H}
$$
holds for all $f,g\in\mathcal{H}$.

The integral operator technique\citep{smale2007learning,caponnetto2007optimal} consider the covariance operator on the Hilbert space $\mathcal{H}$ defined as $\Sigma = \mathbb{E}_{\rho_\mathcal{X}}K_x\otimes K_x$. Then for all $f\in\mathcal{H}$, using the reproducing property, we know that
$$
(\Sigma f)(z)
=\left<K_z,\Sigma f\right>_{\mathcal{H}}=\mathbb{E}[f(X)k(X,z)]=\mathbb{E}[f(X)K_x(X)].
$$

If we consider the mapping $S:\mathcal{H}\rightarrow L_2(dx)$ defined as a parameterization of a vast class of functions in $\mathbb{R}^{\mathcal{X}}$ via $\mathcal{H}$ through the mapping $(Sg)(x)=\left<g,K_x\right>$ ($\Phi(x)=K_x=K(\cdot,x)$). Its adjoint operator $S^\ast$ then can be  defined as $S^\ast:\mathcal{L}_2\rightarrow\mathcal{H}: g\rightarrow \int_\mathcal{X}g(x)K_x \rho_X(dx)$ and at the same time $\Sigma$ is the same as the self-adjoint operator $S^\ast S$ and the self-adjoint operator $\mathcal{L}=SS^\ast:L_2\rightarrow L_2$ can be defined as
$$
    (\mathcal{L}f)(x) =(SS^*f)(x)=\int_\mathcal{X} K(x,z)f(z)d\rho_{\mathcal{X}}(x),\forall f\in L_2
    $$

Next we consider the eigen-decomposition of the integral operator $\mathcal{L}$ via Mecer's Theorem. There exists an orthonormal basis $\{\psi_i\}$ of $\mathcal{L}_2(\mathcal{X})$ consisting of eigenfunctions of kernel integral operator $\mathcal{L}$. At the same time, the kernel function have the following representation $K(s,t)=\sum_{i=1}^\infty \lambda_i e_i(s)e_j(t)$ where $e_i$ are orthogonal basis of $\mathcal{L}_2(\rho_\mathcal{X})$. Then $e_i$ is also the eigenvector of the covariance operator $\Sigma$ with eigenvalue $\lambda_i>0$, \emph{i.e.} $\Sigma e_i = \lambda_i e_i$.


\section{Proof of the Upper Bound}

In this section, we consider the convergence of the gradient descent algorithm to the target function \ref{eq:objective}. In particular, we consider the gradient descent as a special case of a wider class of spectral filter algorithms \citep{dieuleveut2016nonparametric,gerfo2008spectral,pillaud2018statistical,lin2020optimal}. In our inverse problem setting, the spectral filter is defined as the estimator of the following form for $\lambda>0$,
$$
\hat q_\lambda=g_\lambda(\hat \Sigma_{Id,\mathcal{A}_1})\mathcal{A}_2\hat S_n^\ast \hat y,
$$
where $\hat S_n g=\left(g(x_1),\cdot,g(x_n)\right)$ (leads to $\hat S_n^\ast$ maps from $\mathbb{R}^n$ to $\mathcal{H}$ via $\hat S_n^\ast(a_1,a_2,\cdots,a_n)=\frac{1}{n}\sum_{i=1}^na_nK_{x_n}$), $\hat\Sigma_{\mathcal{O}_1,\mathcal{O}_2}=\frac{1}{n}\sum_{i=1}^n \mathcal{O}_1K_x\otimes \mathcal{O}_2K_x$ and $Id$ is the identity operator. The function $q_\lambda:\mathbb{R}^+\rightarrow\mathbb{R}^+$ is a function known as \emph{filter}, which is an approximation of $x^{-1}$ controlled by $\lambda$. We further define the error of approximation via $r_\lambda(x)=1-xq_\lambda(x)$. \textbf{Spectral Filters} need the function $q_\lambda$ further satisfies 

$$
\lambda q_\lambda(x)\le c_q, r_\lambda(x)x^u\le c_q\lambda^u, \forall x>0,\lambda>0,u\in[0,1],
$$
for some positive $c_q>0$. Next we show that the averaged  gradient descent can be considered as spectral filter algorithm with filter $q^\eta(x)=\left(1-\frac{1-(1-\gamma x)^t}{\gamma t x}\right)\frac{1}{x}$. Let us consider the  gradient descent
$\eta_0=0, \eta_u=\eta_{u-1}+\gamma (\mathcal{A}_2^\top\hat S_n^\ast \hat y-\hat \Sigma_{Id,\mathcal{A}_1}\eta_{t-1})$, then
\begin{equation}
\begin{aligned}
   \eta_t&=(I-\gamma \hat \Sigma_{Id,\mathcal{A}_1})\eta_{t-1}+\gamma \mathcal{A}_2^\top\hat S_n^\ast \hat y=\gamma\sum_{k=0}^{t-1}(I-\gamma \hat \Sigma_{Id,\mathcal{A}_1}))^k \mathcal{A}_2^\top\hat S_n^\ast \hat y\\
   &=\left[I-(I-\gamma \hat \Sigma_{Id,\mathcal{A}_1}))^t\right](\hat \Sigma_{Id,\mathcal{A}_1}))^{-1}\mathcal{A}_2^\top\hat S_n^\ast \hat y
\end{aligned}
\end{equation}
and
\begin{equation}
\begin{aligned}
   \bar\eta_t&=\frac{1}{t}\sum_{i=0}^t\eta_i=\frac{1}{t}\sum_{i=0}^t \left[I-(I-\gamma \hat \Sigma_{Id,\mathcal{A}_1}))^i\right](\hat \Sigma_{Id,\mathcal{A}_1}))^{-1}\mathcal{A}_2^\top\hat S_n^\ast \hat y.
\end{aligned}
\end{equation}
Thus if we take the filter $q_t(x)=\frac{1}{x}\left(1-\frac{1-(1-\gamma x)^t}{\gamma t x}\right)$, we can have $\bar\eta=q_t(\Sigma_{Id,\mathcal{A}_1})\mathcal{A}_2^\top\hat S_n^\ast \hat y$. At the same time $x^ur_t(x)=x^u(1-xq_t(x))=x^u\frac{1-(1-\gamma t)^t}{\gamma t x}\le\frac{(\gamma t x)^{1-u}}{\gamma t x}x^u=\frac{1}{(\gamma t)^u}$. Thus we can consider the gradient descent algorithm for the inverse problem as a spectral filtering algorithm.

Next, we compare the spectral filter of early stopped gradient descent with ridge regression and decompose the risk to bias and variance terms. Via bounding the bias and variance separately, we can achieve information theoretical optimal upper bound for such problems.

\subsection{Convergence Of the Gradient Descent Algorithm}

\label{appendix:gd}

To conduct our proof of the upper bound, we consider $g_\lambda = (\Sigma_{Id,\mathcal{A}_1}+\lambda I)^{-1}\mathcal{A}_2S^\ast f_\rho$ and decompose the error as $\underbrace{(g_\lambda-u^*)}_{\text{Bias}}+\underbrace{(\hat q_\lambda-g_\lambda)}_{\text{Variance}}$. We first bound the bias in the general Sobolev norm then come to bound the variance.

\label{appendix:gd}

\subsubsection{Auxiliary Lemmas}

We first introduce several auxiliary lemmas which aims to bound different quantities according to the effective dimension/capacity of the kernel covariance operator. We define $\mathcal{N}(\lambda)=\mathbb{E}_x||(\Sigma_{\mathcal{A}_1}+\lambda)^{-1/2}\mathcal{A}_2K_x||_H^2=\text{Tr}((\Sigma_{\mathcal{A}_1}+\lambda)^{-1}\Sigma_{\mathcal{A}_2\mathcal{A}_2})$, $\mathcal{N}_\infty^1(\lambda)=\sup_{x\in\rho(x)}||(\Sigma_{Id,\mathcal{A}_1}+\lambda)^{-1/2}\mathcal{A}_2K_x||_H^2$ and $\mathcal{N}_\infty^2(\lambda)=\sup_{x\in\rho(x)}||(\Sigma_{Id,\mathcal{A}_1}+\lambda)^{-1/2}K_x||_H^2$ which are 
important important quantities used to bound the variance of our estimator.

\begin{lemma}There exists a constant $D$ such that the following inequality is satisfied for $\lambda>0$,
$$
\mathcal{N}(\lambda)\le D (\lambda)^{-\frac{1}{p+\alpha}+\frac{p-2q}{p+\alpha}}
$$
\end{lemma}

\begin{proof} We use the spectral representation to bound the effective dimension $\mathcal{N}(\lambda)$ as
\begin{equation}
    \begin{aligned}
       \mathcal{N}(\lambda)&=\text{Tr}((\Sigma_{\mathcal{A}_1}+\lambda)^{-1}\Sigma_{\mathcal{A}_2\mathcal{A}_2})=\sum_{i=1}^\infty \frac{\lambda_iq_i^2}{\lambda_ip_i+\lambda}\\
       &\lesssim \sum_{i=1}^\infty \frac{i^{-\alpha-2q}}{i^{-\alpha-p}+\lambda}\le \int_0^\infty \frac{\tau^{p-2q}}{1+\lambda(\tau^{\alpha+p})}d\tau\\
       &=(\lambda)^{-\frac{1}{p+\alpha}}\int_0^\infty\frac{(\lambda)^{\frac{2q-p}{p+\alpha}}\tau^{p-q}}{1+\tau^{\alpha+p}}=\Omega\left((\lambda)^{-\frac{1}{p+\alpha}-\frac{p-2q}{p+\alpha}}\right)
    \end{aligned}
\end{equation}
\end{proof}

\begin{lemma}\label{lemma:DOF}
There exists a constant $D$ such that the following inequality is satisfied for $\lambda>0$,
$$
\text{Trace}((\Sigma_{\mathcal{A}_1}+\lambda)^{-1}\Sigma_{Id,Id})\le D \lambda^{\frac{-p-1}{p+\alpha}}
$$
\end{lemma}

\begin{proof} Similarly we use the spectral representation to bound the LHS as
\begin{equation}
    \begin{aligned}
       \text{Trace}((\Sigma_{\mathcal{A}_1}+\lambda)^{-1}\Sigma_{Id,Id})&=\sum_{i=1}^\infty \frac{\lambda_i}{\lambda_ip_i+\lambda}\\
       &\lesssim \sum_{i=1}^\infty \frac{i^{-\alpha}}{i^{-\alpha-p}+\lambda}\le \int_0^\infty \frac{\tau^{p}}{1+\lambda(\tau^{\alpha+p})}d\tau\\
       &=(\lambda)^{-\frac{1}{p+\alpha}}\int_0^\infty\frac{(\lambda)^{\frac{-p}{p+\alpha}}\tau^{p}}{1+\tau^{\alpha+p}}=\Omega\left(\lambda^{\frac{-p-1}{p+\alpha}}\right)
    \end{aligned}
\end{equation}
\end{proof}

\begin{lemma}\label{lemma:hx} We denote the following quantity by $\mathcal{N}^1_\infty$, $\mathcal{N}^2_\infty$ and $\mathcal{N}^3_\infty$ can be bounded by
\begin{itemize}
    \item $\mathcal{N}^1_\infty(\lambda)=\sup_{x\in\rho(x)}||(\Sigma_{Id,\mathcal{A}_1}+\lambda)^{-1/2}K_x||_H^2\le ||k_v^{\alpha}||_\infty^2\lambda^{-\frac{\mu\alpha+p}{\alpha+p}},$
    \item $\mathcal{N}^2_\infty(\lambda)=\sup_{x\in\rho(x)}||(\Sigma_{Id,\mathcal{A}_1}+\lambda)^{-1/2}\mathcal{A}_2K_x||_H^2\le ||k_v^{\alpha}||_\infty^2\lambda^{-\frac{\mu\alpha+p+2q}{\alpha+p}},$
    \item $\mathcal{N}^3_\infty(\lambda)=\sup_{x\in\rho(x)}||(\Sigma_{Id,\mathcal{A}_1}+\lambda)^{-1/2}\mathcal{A}_1K_x||_H^2\le ||k_v^{\alpha}||_\infty^2\lambda^{-\frac{\mu\alpha+3p}{\alpha+p}}.$
\end{itemize}
\end{lemma}

\begin{proof}
We can also prove the bound from the spectral formulation and the $l_\infty$ embedding property of the kernel function

\begin{equation}
    \begin{aligned}
    ||(\Sigma_{Id,\mathcal{A}_1}+\lambda)^{-1/2}K_x||_H^2&=\sum_{i\ge 1}\frac{\lambda_i}{\lambda_ip_i+\lambda}e_i^2(x)\\
    &\le\left(\sum_{i\ge 1}\lambda_i^{\mu}e_i^2(x)\right)\sup_{i\ge 1}\frac{\lambda_i^{1-\mu}}{\lambda_ip_i+\lambda}\lesssim\left(\sum_{i\ge 1}\lambda_i^{\mu}e_i^2(x)\right)\sup_{i\ge 1}\frac{i^{-(1-\mu)\alpha}}{i^{-\alpha-p}+\lambda}\\
    &\le \lambda^{-\frac{\mu\alpha+p}{\alpha+p}}||k_v^\mu||_\infty^2,
    \end{aligned}
\end{equation}
and
\begin{equation}
    \begin{aligned}
    ||(\Sigma_{Id,\mathcal{A}_1}+\lambda)^{-1/2}\mathcal{A}_2K_x||_H^2&=\sum_{i\ge 1}\frac{\lambda_iq_i^2}{\lambda_ip_i+\lambda}e_i^2(x)\\
    &\le\left(\sum_{i\ge 1}\lambda_i^{\mu}e_i^2(x)\right)\sup_{i\ge 1}\frac{\lambda_i^{1-\mu}q_i^2}{\lambda_ip_i+\lambda}\lesssim\left(\sum_{i\ge 1}\lambda_i^{\mu}e_i^2(x)\right)\sup_{i\ge 1}\frac{i^{-(1-\mu)\alpha-2q}}{i^{-\alpha-p}+\lambda}\\
    &\le \lambda^{-\frac{\mu\alpha+p+2q}{\alpha+p}}||k_v^\mu||_\infty^2.
    \end{aligned}
\end{equation}
Similarly we have
\begin{equation}
    \begin{aligned}
    ||(\Sigma_{Id,\mathcal{A}_1}+\lambda)^{-1/2}\mathcal{A}_1K_x||_H^2&=\sum_{i\ge 1}\frac{\lambda_ip_i^2}{\lambda_ip_i+\lambda}e_i^2(x)\\
    &\le\left(\sum_{i\ge 1}\lambda_i^{\mu}e_i^2(x)\right)\sup_{i\ge 1}\frac{\lambda_i^{1-\mu}p_i^2}{\lambda_ip_i+\lambda}\lesssim\left(\sum_{i\ge 1}\lambda_i^{\mu}e_i^2(x)\right)\sup_{i\ge 1}\frac{i^{-(1-\mu)\alpha-2p}}{i^{-\alpha-p}+\lambda}\\
    &\le \lambda^{-\frac{\mu\alpha+2p}{\alpha+p}}||k_v^\mu||_\infty^2.
    \end{aligned}
\end{equation}
\end{proof}

\begin{lemma}
\label{lemma:soblevnormsigma} For all $\lambda>0$, we have 
$$
||\Sigma^{\frac{1-\gamma}{2}}(\Sigma_{Id,\mathcal{A}_1}+\lambda)^{-1/2}||^2\le \lambda ^{-\frac{\gamma\alpha+p}{\alpha+p}}
$$
\end{lemma}

\begin{proof}
We first bound $||\Sigma^{\frac{1-\gamma}{2}}(\Sigma_{Id,\mathcal{A}_1}+\lambda)^{-1/2}||^2$ 
\begin{equation*}
    \begin{aligned}
       ||\Sigma^{\frac{1-\gamma}{2}}(\Sigma_{Id,\mathcal{A}_1}+\lambda)^{-1/2}||^2=\sup_{i\ge 1}\frac{\lambda_i^{1-\gamma}}{\lambda_ip_i+\lambda}\lesssim\sup_{i\ge 1}\frac{i^{-(1-\gamma)\alpha}}{i^{-\alpha-p}+\lambda}\le \lambda ^{-\frac{\gamma\alpha+p}{\alpha+p}}
    \end{aligned}
\end{equation*}
\end{proof}

\begin{lemma}
\label{lemma:aba} With probability $1-e^{-\tau}$, we have
$$
||(\Sigma_{Id,\mathcal{A}_1}+\lambda )^{-1/2}(\hat \Sigma_{Id,\mathcal{A}_1}-\Sigma_{Id,\mathcal{A}_1} )(\Sigma_{Id,\mathcal{A}_1}+\lambda )^{-1/2}||^2\lesssim \sqrt{\frac{\tau\sqrt{\mathcal{N}_\infty^1(\lambda)\mathcal{N}_\infty^3(\lambda)}}{n}}
$$
and as a consequence once {$n\gtrsim \tau\lambda^{-\frac{\mu\alpha+2p}{\alpha+p}}$}, we'll have
$$
\frac{1}{2}\le||(\Sigma_{Id,\mathcal{A}_1}+\lambda )^{1/2}(\hat\Sigma_{Id,\mathcal{A}_1}+\lambda )^{-1/2}||\le 2, \frac{1}{2}\le||(\Sigma_{Id,\mathcal{A}_1}+\lambda )^{-1/2}(\hat\Sigma_{Id,\mathcal{A}_1}+\lambda )^{1/2}||\le 2.
$$

\end{lemma}
\begin{proof}
 We utilize the concentration result for Hilbert space valued random variable \citep{pinelis1985remarks} to prove the bound here. Now, we consider the operator $C_x:\mathcal{H}\rightarrow\mathcal{H}$ the operator defined by
 $$
 C_xf:=\mathcal{A}_1f(x)k(x,\cdot) =\left<f,\mathcal{A}_1K_x\right>K_x,
 $$
 and consider the random variable $\xi_x:=(\Sigma_{Id,\mathcal{A}_1}+\lambda )^{1/2}C_x(\Sigma_{Id,\mathcal{A}_1}+\lambda )^{-1/2}$. From definition, we know that
\begin{equation}
    \begin{aligned}
       \xi_xf&=(\Sigma_{Id,\mathcal{A}_1}+\lambda )^{1/2}C_x(\Sigma_{Id,\mathcal{A}_1}+\lambda )^{-1/2}f\\
       &=\left<f,(\Sigma_{Id,\mathcal{A}_1}+\lambda )^{-1/2}\mathcal{A}_1K_x\right>(\Sigma_{Id,\mathcal{A}_1}+\lambda )^{1/2}K_x\\
       &=((\Sigma_{Id,\mathcal{A}_1}+\lambda )^{1/2}K_x\otimes (\Sigma_{Id,\mathcal{A}_1}+\lambda )^{-1/2}\mathcal{A}_1K_x)f.
    \end{aligned}
\end{equation}
At the same time, we know that $||f\otimes g||=||f||_\mathcal{H}||g||_\mathcal{H}$ for all $f,g\in\mathcal{H}$, thus utilizing the concentration results for Hilbert space valued random variable, we have
$$
||(\Sigma_{Id,\mathcal{A}_1}+\lambda )^{-1/2}(\hat \Sigma_{Id,\mathcal{A}_1}-\Sigma_{Id,\mathcal{A}_1} )(\Sigma_{Id,\mathcal{A}_1}+\lambda )^{-1/2}||^2\lesssim \sqrt{\frac{\tau\sqrt{\mathcal{N}_\infty^1(\lambda)\mathcal{N}_\infty^3(\lambda)}}{n}}.
$$
From Lemma \ref{lemma:hx}, we know that $\mathcal{N}^1_\infty(\lambda)=\sup_{x\in\rho(x)}||(\Sigma_{Id,\mathcal{A}_1}+\lambda)^{-1/2}K_x||_H^2\le ||k_v^{\alpha}||_\infty^2\lambda^{-\frac{\mu\alpha+p}{\alpha+p}},$ and $\mathcal{N}^3_\infty(\lambda)=\sup_{x\in\rho(x)}||(\Sigma_{Id,\mathcal{A}_1}+\lambda)^{-1/2}\mathcal{A}_1K_x||_H^2\le ||k_v^{\alpha}||_\infty^2\lambda^{-\frac{\mu\alpha+3p}{\alpha+p}}.$ Thus once {$n\gtrsim \tau\lambda^{-\frac{\mu\alpha+2p}{\alpha+p}}$}, we'll have
$$
\frac{1}{2}\le||(\Sigma_{Id,\mathcal{A}_1}+\lambda )^{1/2}(\hat\Sigma_{Id,\mathcal{A}_1}+\lambda )^{-1/2}||\le 2, \frac{1}{2}\le||(\Sigma_{Id,\mathcal{A}_1}+\lambda )^{-1/2}(\hat\Sigma_{Id,\mathcal{A}_1}+\lambda )^{1/2}||\le 2.
$$
\end{proof}
\begin{theorem}[Bernstein's Inequality] \label{thm:bernstein} Let $(\Omega,\mathcal{B},P)$ be a probability space, $H$ be a separable Hilbert space, and $\xi:\Omega\rightarrow H$ with
$$
\mathbb{E}_P||\xi||_H^m\le\frac{1}{2}m!\sigma^2L^{m-2}
$$
for all $m\ge 2$. Then, for $\tau\ge 1$ and $n\ge 1$, the following concentration inequality is satisfied
$$
\mathcal{P}^n\left[\left||\frac{1}{n}\sum_{i=1}^n\xi(\omega_i)-\mathbb{E}_P\xi\right||_H^2\ge32\frac{\tau^2}{n}\left(\sigma^2+\frac{L^2}{n}\right)\right]\le 2e^{-\tau}
$$

\end{theorem}

\begin{lemma} [Lemma 25 in \citep{fischer2020sobolev}]
\label{lemma:basic}
For $\lambda>0$ and $0\le\alpha\le 1$, the function $f_{\lambda,\alpha}:[0,\infty)\rightarrow\mathbb{R}$ be defined by $f_{\lambda,\alpha}(t):=\frac{t^{\alpha}}{\lambda+t}$. In the case $\alpha=0$ the function is decreasing and for $\alpha=1$ the function is increasing. Furthermore
$$
\lambda^{\alpha-1}/2\le\sup_{t\ge0}f_{\lambda,\alpha}(t)\le\lambda^{\alpha-1}
$$
for $0<\alpha<1 $ the function attain its supremum at $t^\ast=\frac{\lambda\alpha}{1-\alpha}$
\end{lemma}
\begin{proof}
For completeness, we provide the proof here. For function $f_{\lambda,\alpha}(t):=\frac{t^{\alpha}}{\lambda+t}$ with $0<\alpha<1$, we know the derivative of it is $f^\prime_{\lambda,\alpha}(t)=\frac{\alpha t^{\alpha-1}(\alpha+t)-t^\alpha}{(\lambda+t)^2}$. The derivative $f^\prime_{\lambda,\alpha}$ has a unique root at $t^\ast=\alpha\lambda/(1-\alpha)$. $f_{\lambda,\alpha}$ attains global maximum at $t^\ast$ and
$$
\sup_{t\ge 0}f_{\lambda,\alpha}(t)=f_{\lambda,\alpha}(t^\ast)=\lambda^{\alpha-1}\alpha^\alpha(1-\alpha)^{1-\alpha}\le \lambda^{\alpha-1}.
$$
At the same time, $(\alpha^\alpha(1-\alpha)^{1-\alpha})^\prime=\alpha^\alpha(1-\alpha)^{1-\alpha}\log\left(\frac{\alpha}{1-\alpha}\right)$ thus $\alpha^\alpha(1-\alpha)^{1-\alpha}$ achieves minimum $\frac{1}{2}$ when $\alpha=\frac{1}{2}$. Thus we know $\lambda^{\alpha-1}/2\le\sup_{t\ge0}f_{\lambda,\alpha}(t)$.
\end{proof}

\subsubsection{Bias}
\label{section:appedixbias}

In this section, we consider the bias introduced by the regularization factor, \emph{i.e.} the difference between $g_\lambda=(\Sigma_{Id,\mathcal{A}_1}+\lambda I)^{-1}\mathcal{A}_2 S^* f_\rho$ and the ground truth solution $\mathcal{A}_1^{-1}\mathcal{A}_2f_\rho$.  
\begin{lemma} \label{lemma:bias} If $u^\ast=\mathcal{A}_1^{-1}\mathcal{A}_2f_\rho\in [H]^\beta$ holds, then for all $0\le\gamma\le\beta$ and $\lambda>0$, the following bounds holds
$$
||g_\lambda-\mathcal{A}_1^{-1}\mathcal{A}_2f_\rho||_\gamma \lesssim \lambda^{\frac{(\frac{\beta-\gamma}{2})\alpha}{\alpha+p}}||u^\ast||_{[H]^\beta}.
$$
Here $g_\lambda=(\Sigma_{Id,\mathcal{A}_1}+\lambda I)^{-1}\mathcal{A}_2 S^* f_\rho$.
\end{lemma}

\begin{proof} Since $u^\ast=\mathcal{A}_1^{-1}\mathcal{A}_2f_\rho\in [H]^\beta$, we can use the spectral representation $u^\ast=\sum_{i=1}^n a_ie_i$ with $||u^\ast||_{[H]^\beta}=\sum_{i=1}^\infty \lambda_i^{-\beta}a_i$. At the same time $\mathcal{A}_2f_\rho=\mathcal{A}_1u^\ast=\sum_{i=1}^n a_ip_ie_i$. We also observe that the matrix $(\Sigma_{Id,\mathcal{A}_1}+\lambda I)^{-1}$ have the spectral representation $(\Sigma_{Id,\mathcal{A}_1}+\lambda I)^{-1}=\sum_{i=1}^\infty({\lambda_i p_i+\lambda})^{-1} e_i\otimes e_i$ and leads to the spectral representation of the solution

$$
g_\lambda=(\Sigma_{Id,\mathcal{A}_1}+\lambda I)^{-1}\mathcal{A}_2 S^* f_\rho =\sum_{i=1}^\infty\frac{\lambda_iq_i}{\lambda_i p_i+\lambda}\frac{p_i}{q_i}a_ie_i= \sum_{i=1}^\infty\frac{\lambda_ip_i}{\lambda_i p_i+\lambda}a_ie_i
$$

Then we can bound the bias via the spectral representation 

\begin{equation}
    \begin{aligned}
    ||g_\lambda-\mathcal{A}_1^{-1}\mathcal{A}_2f_\rho||_\gamma^2&=|| (\Sigma_{Id,\mathcal{A}_1}+\lambda I)^{-1}\mathcal{A}_2S^\ast f_\rho-\mathcal{A}_1^{-1}\mathcal{A}_2f_\rho||_\gamma^2
    \\&=\left||\sum_{i=1}^\infty\frac{\lambda_ip_i}{\lambda_i p_i+\lambda}a_ie_i-a_ie_i\right||^2=\left||\sum_{i=1}^\infty\frac{\lambda}{\lambda_i p_i+\lambda}a_ie_i\right||_\gamma^2
    \\&=\sum_{i=1}^\infty \left(\frac{\lambda}{\lambda_i p_i+\lambda}a_i\right)^2\lambda_i^{-\gamma}
    \\&=\lambda^2\left(\sup_{i\ge 1}\frac{i^{-\alpha(\frac{\beta-\gamma}{2})}}{\lambda+i^{-\alpha-p}}\right)^2\sum_{i\ge 1}\lambda_i^{-\beta}a_i^2\le \lambda^{\frac{({\beta-\gamma})\alpha}{\alpha+p}}||u^\ast||_{[H]^\beta}^2
    \end{aligned}
\end{equation}

\end{proof}

In this section, we also bound a bias over the energy function $||\mathcal{A}_1 g_\lambda-\mathcal{A}_2 f_\rho||_2^2$, which will be used in bounding the variance term.

\begin{lemma} \label{lemma:biasenergy}
 If $u^\ast=\mathcal{A}_1^{-1}\mathcal{A}_2f_\rho\in [H]^\beta$ holds, then for all $0\le\gamma\le\beta$ and $\lambda>0$, the following bounds holds
$$
||\mathcal{A}_1g_\lambda-\mathcal{A}_2f_\rho||_2 \lesssim \lambda^{\frac{\beta\alpha-2p}{2(\alpha+p)}}||u^\ast||_{[H]^\beta}.
$$
Here $g_\lambda=(\Sigma_{Id,\mathcal{A}_1}+\lambda I)^{-1}\mathcal{A}_2 S^* f_\rho$.
\end{lemma}

\begin{proof} As discussed in the proof of Lemma \ref{lemma:bias}, we have the spectral representation of $g_\lambda$ as
$$
g_\lambda=(\Sigma_{Id,\mathcal{A}_1}+\lambda I)^{-1}\mathcal{A}_2 S^* f_\rho =\sum_{i=1}^\infty\frac{\lambda_iq_i}{\lambda_i p_i+\lambda}\frac{p_i}{q_i}a_ie_i= \sum_{i=1}^\infty\frac{\lambda_ip_i}{\lambda_i p_i+\lambda}a_ie_i
$$

Thus $\mathcal{A}_1g_\lambda-\mathcal{A}_2f_\rho=\sum_{i=1}^\infty\left(\frac{\lambda_ip_i^2}{\lambda_i p_i+\lambda}-p_i\right)a_ie_i=-\sum_{i=1}^\infty\left(\frac{p_i\lambda}{\lambda_i p_i+\lambda}\right)a_ie_i$ and we can have the bound of the bias in the energy norm as

\begin{equation*}
    \begin{aligned}
       ||\mathcal{A}_1g_\lambda-\mathcal{A}_2f_\rho||_2^2 &=\left||\sum_{i=1}^\infty\left(\frac{p_i\lambda}{\lambda_i p_i+\lambda}-p_i\right)a_ie_i\right||_2^2\\
       &=\sum_{i=1}^\infty\left(\frac{p_i\lambda}{\lambda_i p_i+\lambda}a_i\right)^2=\lambda^2\left(\sup_{i\ge 1}\frac{i^{-(\frac{\alpha\beta}{2})-p}}{\lambda+i^{-\alpha-p}}\right)^2\sum_{i\ge 1}\lambda_i^{-\beta}a_i^2\\
       &\lesssim\lambda^{\frac{\beta\alpha-2p}{\alpha+p}}||u^\ast||_{[H]^\beta}.
    \end{aligned}
\end{equation*}

\end{proof}

\subsubsection{Variance}
\label{section:appedixvaraince}

In this section, we bound the variance which defined as the difference between  between $g_\lambda = (\Sigma_{Id,\mathcal{A}_1}+\lambda I)^{-1}\mathcal{A}_2S^\ast f_\rho$ and $\hat g_\lambda =q_\lambda(\hat \Sigma_{Id,\mathcal{A}_1})\mathcal{A}_2\hat S^\ast y$ at the scale $
O\left(\frac{(\sigma^2+R^2\lambda^{2r})\mathcal{N}(\lambda_q)}{n}+\frac{\lambda^{\frac{(\beta-\gamma)\alpha-\mu\alpha-p}{\alpha+p}}}{n}+o(\frac{1}{n})\right)$. We first did the following decomposition
\begin{equation}
    \begin{aligned}
    \Sigma^{\frac{1-\gamma}{2}}(g_\lambda-\hat g_\lambda)&=\Sigma^{\frac{1-\gamma}{2}}q_\lambda(\hat \Sigma_{Id,\mathcal{A}_1})(\mathcal{A}_2\hat S^\ast y-(\hat \Sigma_{Id,\mathcal{A}_1}) g_\lambda)+\Sigma^{\frac{1-\gamma}{2}}\left[g_\lambda(\hat \Sigma_{Id,\mathcal{A}_1})\hat \Sigma_{Id,\mathcal{A}_1}-I\right]g_\lambda\\
    &=\Sigma^{\frac{1-\gamma}{2}}q_\lambda(\hat \Sigma_{Id,\mathcal{A}_1})(\Sigma_{Id,\mathcal{A}_1}^\lambda) ^{1/2}\left[\frac{1}{n}\sum_{i=1}^n(\xi(x_i,y_i))\right]+\Sigma^{\frac{1-\gamma}{2}} r(\hat \Sigma_{Id,\mathcal{A}_1}) g_\lambda\\
     &=\underbrace{\Sigma^{\frac{1-\gamma}{2}}q_\lambda(\hat \Sigma_{Id,\mathcal{A}_1})(\Sigma_{Id,\mathcal{A}_1}^\lambda) ^{1/2}\left[\frac{1}{n}\sum_{i=1}^n(\xi(x_i,y_i)-\mathbb{E}_P\xi(x,y))\right]}_{(\uppercase\expandafter{\romannumeral1})}\\
     &+\underbrace{\Sigma^{\frac{1-\gamma}{2}}q_\lambda(\hat \Sigma_{Id,\mathcal{A}_1})(\Sigma_{Id,\mathcal{A}_1}^\lambda) ^{1/2}\mathbb{E}_P\xi(x,y)}_{(\uppercase\expandafter{\romannumeral2})}+\underbrace{\Sigma^{\frac{1-\gamma}{2}} r(\hat \Sigma_{Id,\mathcal{A}_1}) g_\lambda}_{(\uppercase\expandafter{\romannumeral3})},
    \end{aligned}
\end{equation}
where we take the random variable $\xi(x,y)$ as $\xi(x,y)=(\Sigma_{Id,\mathcal{A}_1}+\lambda)^{-1/2}(y\mathcal{A}_2 K_x-\mathcal{A}_1 g_\lambda(x)  K_x)$ which satisfies $\mathbb{E}_Q\xi_2=(\Sigma_{Id,\mathcal{A}_1}+\lambda)^{-1/2}(\mathcal{A}_2 f_Q-\Sigma_{Id,\mathcal{A}_1}^Q g_\lambda)$ where $f_Q=\mathbb{E}_Qf(x)K_x$ and $\Sigma_{Id,\mathcal{A}_1}^Q=\mathbb{E}_Q K_x\otimes \mathcal{A}_1 K_x$ for arbitrary distribution $Q$ and $\mathbb{E}_{\mathbb{P}} \xi(x,y)=(\Sigma_{Id,\mathcal{A}_1}+\lambda)^{-1/2}(\mathcal{A}_2S^\ast f_\rho-\Sigma_{Id,\mathcal{A}_1} g_\lambda$). We bound different terms (\uppercase\expandafter{\romannumeral1}), (\uppercase\expandafter{\romannumeral2}) and (\uppercase\expandafter{\romannumeral3}) separately and combine them to get the final upper bound. We show that (\uppercase\expandafter{\romannumeral1}) is the mean variance term and is at the scale $\frac{\mathcal{N}(\lambda)}{n}=\frac{\text{Tr}(\Sigma_{Id,\mathcal{A}_{1}}+\lambda)^{-1}\Sigma_{\mathcal{A}_2,\mathcal{A}_2}}{n}$when the problem is regular. Term (\uppercase\expandafter{\romannumeral2}) and (\uppercase\expandafter{\romannumeral3}) is smaller than the bias. Our bound of term (\uppercase\expandafter{\romannumeral3}) bounds tighter than \citep{pillaud2018statistical} (the second term, Lemma 10) via the spectral representation.

\paragraph{Bounding term (\uppercase\expandafter{\romannumeral1}).} The term (\uppercase\expandafter{\romannumeral1}) is the concentration error of the random variable $\xi(x,y)$ and can be bounded via a Bernstein Inequality. We first bound term (\uppercase\expandafter{\romannumeral1}) via the following decomposition
\begin{equation*}\small
  \begin{aligned}
     &\left||\Sigma^{\frac{1-\gamma}{2}}q_\lambda(\hat \Sigma_{Id,\mathcal{A}_1})(\Sigma_{Id,\mathcal{A}_1}^\lambda) ^{1/2}\left[\frac{1}{n}\sum_{i=1}^n(\xi(x_i,y_i)-\mathbb{E}_P\xi(x,y))\right]\right||_H^2\le    ||\Sigma^{\frac{1-\gamma}{2}}(\Sigma_{Id,\mathcal{A}_1}^\lambda )^{-1/2}||^2 \left||\frac{1}{n}\sum_{i=1}^n(\xi(x_i,y_i)-\mathbb{E}_P\xi)\right||_H^2\\
     &\cdot||( \Sigma_{Id,\mathcal{A}_1}^\lambda )^{1/2}(\hat \Sigma_{Id,\mathcal{A}_1}^\lambda )^{-1/2}||^2||(\hat \Sigma_{Id,\mathcal{A}_1}^\lambda )^{1/2}q_\lambda(\hat \Sigma_{Id,\mathcal{A}_1})(\hat \Sigma_{Id,\mathcal{A}_1}^\lambda )^{1/2}||^2||(\hat \Sigma_{Id,\mathcal{A}_1}^\lambda )^{-1/2}( \Sigma_{Id,\mathcal{A}_1}^\lambda )^{1/2}||^2,
  \end{aligned}
\end{equation*}
where $\Sigma_{Id,\mathcal{A}_1}^\lambda=\Sigma_{Id,\mathcal{A}_1}+\lambda I$ and $\hat \Sigma_{Id,\mathcal{A}_1}^\lambda=\hat \Sigma_{Id,\mathcal{A}_1}+\lambda I$ . At the same time, we knows $||\Sigma^{\frac{1-\gamma}{2}}(\Sigma_{Id,\mathcal{A}_1}+\lambda )^{-1/2}||^2\le\lambda^{-\frac{\gamma\alpha+p}{\alpha+p}}$ (From lemma \ref{lemma:soblevnormsigma}) and $||( \Sigma_{Id,\mathcal{A}_1}^\lambda )^{1/2}(\hat \Sigma_{Id,\mathcal{A}_1}^\lambda )^{-1/2}||^2\le 2$ (From lemma \ref{lemma:aba}) with high probability. At the same time, we have
$$
||(\hat \Sigma_{Id,\mathcal{A}_1}^\lambda )^{1/2}q_\lambda(\hat \Sigma_{Id,\mathcal{A}_1})(\hat \Sigma_{Id,\mathcal{A}_1}^\lambda )^{1/2}||=\sup_{\sigma\in\sigma(\hat \Sigma_{Id,\mathcal{A}_1}^\lambda)} (\sigma+\lambda)q_\lambda(\sigma)\le 2c_q.
$$
Thus we only need to focus on bounding the concentration error $\frac{1}{n}\sum_{i=1}^n(\xi(x_i,y_i)-\mathbb{E}_P\xi)$. We recall the moment condition to control the noise of the observations. There are constants $\sigma,L>0$ such that
$$
\int_{\mathbb{R}}|y-f^*(x)|^mP(dy|x)\le \frac{1}{2}m!\sigma^2 L^{m-2}
$$

is satisfied for $\mu$-almost all $x\in X$ and all $m>2$. Note that the moment condition is satisfied for Gaussian noise with bounded variance or have a bounded observation noise. Then we can bound the second order momentum of the random variable $\xi(x,y)=(\Sigma_{Id,\mathcal{A}_1}+\lambda)^{-1/2}(y\mathcal{A}_2 K_x-\mathcal{A}_1g_\lambda(x) K_x)$ via decomposing the random into three parts $(\Sigma_{Id,\mathcal{A}_1}+\lambda)^{-1/2}(y\mathcal{A}_2 K_x-f^\ast(x) \mathcal{A}_2K_x)$, $(\Sigma_{Id,\mathcal{A}_1}+\lambda)^{-1/2}(f^\ast(x) \mathcal{A}_2K_x-\mathcal{A}_2f^\ast(x) K_x)$ and $(\Sigma_{Id,\mathcal{A}_1}+\lambda)^{-1/2}(\mathcal{A}_2f^\ast(x) K_x-\mathcal{A}_1g_\lambda(x)K_x)$. Base on the decomposition, we can bound the moments of random variable $\xi(x,y)$ as

\begin{align*}
       \mathbb{E}_P||\xi(x,y)||_H^m&=\int\left[||(\Sigma_{Id,\mathcal{A}_1}+\lambda)^{-1/2}\mathcal{A}_2K_x||_H^m\int_{\mathbb{R}}|y-f^*(x)|^mP(dy|x)\right]+\int\left[||(\Sigma_{Id,\mathcal{A}_1}+\lambda)^{-1/2}\mathcal{A}_2K_x||_H^m||f||_{\infty}^m\right]\\
       &+ \int\left[||(\Sigma_{Id,\mathcal{A}_1}+\lambda)^{-1/2}K_x||_H^m\int_{\mathbb{R}}|\mathcal{A}_2f^*(x)-\mathcal{A}_1g_\lambda|^mP(dy|x)\right]
dv(x)\\
&\le \frac{1}{2} m!\sigma^2(L+||f||_{\mathcal{H}})^{m}||h_x^1||_{\mathcal{H}}^{m-2}\text{trace}((\Sigma_{Id,\mathcal{A}_1}+\lambda)^{-1}\Sigma_{\mathcal{A}_2,\mathcal{A}_2})\\
&+||h_x^2||_{\mathcal{H}}^{m-2}||\mathcal{A}_2f^*(x)-\mathcal{A}_1g_\lambda||_{L_\infty}^{m-2}\int |\mathcal{A}_2f^*(x)-\mathcal{A}_1g_\lambda|^2d\mu(x)\\
&\lesssim m!\left(||h_x^1||\right)^{m}\left[\sigma^2\text{trace}((\Sigma_{Id,\mathcal{A}_1}+\lambda)^{-1}\Sigma_{\mathcal{A}_2,\mathcal{A}_2})\right]+m!\left(L_\lambda||h_x^2||\right)^{m-2}\left[||h_x^2||^2||\mathcal{A}_2f^*(x)-\mathcal{A}_1g_\lambda||_2^2\right]\\
\end{align*}
where $L_\lambda=||\mathcal{A}_2f^*(x)-\mathcal{A}_1g_\lambda||_{L_\infty}$, $h_x^1 = (\Sigma_{Id,\mathcal{A}_1}+\lambda)^{-1/2}\mathcal{A}_2K_x$ and $h_x^2 = (\Sigma_{Id,\mathcal{A}_1}+\lambda)^{-1/2}K_x$. The two vectors' norms are bounded in Lemma \ref{lemma:hx} as $\mathcal{N}^1_\infty(\lambda)=\sup_{x\in\rho(x)}||(\Sigma_{Id,\mathcal{A}_1}+\lambda)^{-1/2}K_x||_H^2\le ||k_v^{\alpha}||_\infty^2\lambda^{-\frac{\mu\alpha+p}{\alpha+p}},$ and $\mathcal{N}^2_\infty(\lambda)=\sup_{x\in\rho(x)}||(\Sigma_{Id,\mathcal{A}_1}+\lambda)^{-1/2}K_x||_H^2\le ||k_v^{\alpha}||_\infty^2\lambda^{-\frac{\mu\alpha+p+2q}{\alpha+p}}$. At the same time, we know that $||\mathcal{A}_1g_\lambda-\mathcal{A}_2f_\rho||_2 \lesssim \lambda^{\frac{\beta\alpha-2p}{2(\alpha+p)}}||u^\ast||_{[H]^\beta}$ from Lemma \ref{lemma:biasenergy} and $\text{Trace}((\Sigma_{\mathcal{A}_1}+\lambda)^{-1}\Sigma_{Id,Id})\le D \lambda^{\frac{p-1}{p+\alpha}}$ from Lemma \ref{lemma:DOF}. Then using Bernstein Inequality (Theorem \ref{thm:bernstein}), we knows that with probability $1-2e^{-\tau}$

 \begin{equation}
     \begin{aligned}
     &||\frac{1}{n}\sum_{i=1}^n(\xi(x_i,y_i)-\mathbb{E}_P\xi(x,y))||_H^2\\
     &\lesssim \frac{32\tau^2}{n}\left(\sigma^2\text{trace}((\Sigma_{Id,\mathcal{A}_1}+\lambda)^{-1}\Sigma_{\mathcal{A}_2,\mathcal{A}_2})+||h_x^2||^2||\mathcal{A}_2f^*(x)-\mathcal{A}_1g_\lambda||_2^2+\frac{L_\lambda||h_x^2||+||h_x^1||}{n}\right)\\
     &\lesssim \frac{\tau^2}{n}\left(\sigma^2 (\lambda)^{-\frac{1}{p+\alpha}-\frac{p-2q}{p+\alpha}}+\lambda^{-\frac{\mu\alpha-p}{\alpha+p}}\lambda^{\frac{\alpha\beta-2p}{\alpha+p}}+\frac{L_\lambda||h_x^2||+||h_x^1||}{n}\right)
     \end{aligned}
 \end{equation}
Thus we have the final bound $\left||\Sigma^{\frac{1-\gamma}{2}}q_\lambda(\hat \Sigma_{Id,\mathcal{A}_1})(\Sigma_{Id,\mathcal{A}_1}^\lambda) ^{1/2}\left[\frac{1}{n}\sum_{i=1}^n(\xi(x_i,y_i)-\mathbb{E}_P\xi(x,y))\right]\right||_H^2\le$\\ $\frac{\tau^2}{n}\lambda^{-\frac{\gamma\alpha+p}{\alpha+p}}\left(\sigma^2 (\lambda)^{-\frac{1}{p+\alpha}-\frac{p-2q}{p+\alpha}}+\lambda^{-\frac{\mu\alpha-p}{\alpha+p}}\lambda^{\frac{\alpha\beta-2p}{\alpha+p}}+\frac{L_\lambda||h_x^2||+||h_x^1||}{n}\right).$

\begin{remark}
In this remark, we'll bound the $L_\lambda=||\mathcal{A}_2f^*(x)-\mathcal{A}_1g_\lambda||_{L_\infty}$ here. For the embedding theorem of the $\ell_\infty$, $L_\lambda\le||\mathcal{A}_2f^*(x)-\mathcal{A}_1g_\lambda||_{\mu}\lesssim\lambda^{-\frac{(\mu-\beta)_+\alpha}{\alpha+p}}$. From Lemma \ref{lemma:hx}, we know that $||h_x^1||_H^2\lesssim\lambda^{-\frac{\mu\alpha+p+2q}{\alpha+p}}$ and $||h_x^2||_H^2\lesssim\lambda^{-\frac{\mu\alpha+p}{\alpha+p}}$.
\end{remark}

\paragraph{Bounding term (\uppercase\expandafter{\romannumeral3}).}  At last we bound the term $\Sigma^{\frac{1-\gamma}{2}} r(\hat \Sigma_{Id,\mathcal{A}_1}) g_\lambda$ via the following decomposition

\begin{equation*}
    \begin{aligned}
       ||\Sigma^{\frac{1-\gamma}{2}} r(\hat \Sigma_{Id,\mathcal{A}_1}) g_\lambda||_{\mathcal{H}}&=||\Sigma^{\frac{1-\gamma}{2}}\underbrace{( \Sigma_{Id,\mathcal{A}_1}^\lambda)^{-1/2}( \Sigma_{Id,\mathcal{A}_1}^\lambda)^{1/2} }_{Id}\underbrace{(\hat \Sigma_{Id,\mathcal{A}_1}^\lambda)^{-1/2}(\hat \Sigma_{Id,\mathcal{A}_1}^\lambda)^{1/2} }_{Id}r(\hat \Sigma_{Id,\mathcal{A}_1}^\lambda) (\Sigma_{Id,\mathcal{A}_1}^\lambda) ^{-1} \mathcal{A}_2 S^* f_\rho||_{\mathcal{H}}\\
       &\le ||\Sigma^{\frac{1-\gamma}{2}}( \Sigma_{Id,\mathcal{A}_1})^{-1/2}||||( \Sigma_{Id,\mathcal{A}_1})^{1/2} (\hat \Sigma_{Id,\mathcal{A}_1})^{-1/2}||||(\hat \Sigma_{Id,\mathcal{A}_1})^{1/2} r(\hat \Sigma_{Id,\mathcal{A}_1})|||| (\Sigma_{Id,\mathcal{A}_1}^\lambda) ^{-1} \mathcal{A}_2 S^* f_\rho||,
    \end{aligned}
\end{equation*}

where we use $\Sigma_{Id,\mathcal{A}_1}^\lambda$ to denote $\Sigma_{Id,\mathcal{A}_1}+\lambda I$. From Lemma \ref{lemma:soblevnormsigma} we now that 
 $||\Sigma^{\frac{1-\gamma}{2}}( \Sigma_{Id,\mathcal{A}_1}^\lambda)^{-1/2}||^2\le\lambda^{-\frac{\gamma\alpha+p}{\alpha+p}}$. Then we bound the term $||\Sigma^{\frac{1-\gamma}{2}} r(\hat \Sigma_{Id,\mathcal{A}_1}) g_\lambda||_{\mathcal{H}}$ using $r_\lambda(x)x^u\lesssim \lambda^u$ and get
 
 \begin{equation}
     \begin{aligned}
     ||(\hat \Sigma_{Id,\mathcal{A}_1})^{1/2} r(\hat \Sigma_{Id,\mathcal{A}_1})||=\sup_{\sigma\in\sigma(\hat \Sigma_{Id,\mathcal{A}_1}^\lambda)}(\sigma+\lambda)^{1/2}r_\lambda(\sigma)\le \lambda^{1/2}.
     \end{aligned}
 \end{equation}

At the same time, we can bound $|| (\Sigma_{Id,\mathcal{A}_1}^\lambda) ^{-1} \mathcal{A}_2 S^* f_\rho||$ using the spectral representation
\begin{equation}
    \begin{aligned}
    || (\Sigma_{Id,\mathcal{A}_1}^\lambda) ^{-1} \mathcal{A}_2 S^* f_\rho||^2&=\sum_{i=1}^\infty \frac{\lambda_ip_i^2a_i^2}{(\lambda+\lambda_ip_i)^2}\lesssim \left(\sup_{i\ge 1}\frac{\lambda_ip_i^2\lambda_i^{\beta}}{(\lambda+\lambda_ip_i)^2}\right)\sum_{i\ge 1}\lambda_i^{-\beta}a_i^2\\&\le \left(\lambda^{\frac{(\frac{1-\beta}{2})\alpha+p}{\alpha+p}-1}\right)^2 ||u^\ast||_{[H]^\beta}^2\le \lambda^{\frac{\beta}{\alpha+p}-1}||u^\ast||_{[H]^\beta}^2
    \end{aligned}
\end{equation}
Thus we know that 

$$
||(\hat \Sigma_{Id,\mathcal{A}_1})^{1/2} r(\hat \Sigma_{Id,\mathcal{A}_1})||\lesssim \lambda^{-\frac{\gamma\alpha+p}{2(\alpha+p)}}\lambda^{1/2}\lambda^{\frac{\beta}{2(\alpha+p)}-\frac{1}{2}}\le \lambda^{-\frac{(\beta-\gamma)\alpha}{2(\alpha+p)}},
$$
where the last inequality is because $p<0$ in our assumption.

\paragraph{Bounding term (\uppercase\expandafter{\romannumeral2}).}  In this paragraph, we demonstrate the proof to bound the term $$\Sigma^{\frac{1-\gamma}{2}}q_\lambda(\hat \Sigma_{Id,\mathcal{A}_1})(\Sigma_{Id,\mathcal{A}_1}^\lambda) ^{1/2}\mathbb{E}_P\xi(x,y)=\Sigma^{\frac{1-\gamma}{2}}q_\lambda(\Sigma_{Id,\mathcal{A}_1}^\lambda) ^{1/2}(\hat \Sigma_{Id,\mathcal{A}_1})(\Sigma_{Id,\mathcal{A}_1}+\lambda)^{-1/2}(\mathcal{A}_2S^\ast f_\rho-\Sigma_{Id,\mathcal{A}_1} g_\lambda).$$ Note that $\Sigma_{Id,\mathcal{A}_1}(\Sigma_{Id,\mathcal{A}_1}^{\lambda})^{-1}=I-\lambda(\Sigma_{Id,\mathcal{A}_1}^{\lambda})^{-1}$, thus we knows that $\Sigma^{\frac{1-\gamma}{2}}q_\lambda(\hat \Sigma_{Id,\mathcal{A}_1})(\Sigma_{Id,\mathcal{A}_1}^\lambda) ^{1/2}(\Sigma_{Id,\mathcal{A}_1}+\lambda)^{-1/2}(\mathcal{A}_2S^\ast f_\rho-\Sigma_{Id,\mathcal{A}_1} g_\lambda)=\lambda\Sigma^{\frac{1-\gamma}{2}}q_\lambda(\hat \Sigma_{Id,\mathcal{A}_1})(\Sigma_{Id,\mathcal{A}_1}^\lambda)^{-1}\mathcal{A}_2S^\ast f_\rho$. At the same time, according to our assumption on the spectral filter $q_\lambda$, we know that
$$
||(\Sigma_{Id,\mathcal{A}_1}^\lambda) ^{1/2} q(\hat \Sigma_{Id,\mathcal{A}_1}^\lambda)(\Sigma_{Id,\mathcal{A}_1}^\lambda) ^{1/2}||=\sup_{\sigma\in\sigma(\hat \Sigma_{Id,\mathcal{A}_1}^\lambda)} (\sigma+\lambda)q_\lambda(\sigma)\le 2c_q.
$$
Thus we can bound $\Sigma^{\frac{1-\gamma}{2}}q_\lambda(\hat \Sigma_{Id,\mathcal{A}_1})\mathbb{E}_P\xi(x,y)$ via the following decomposition
\begin{equation*}
\begin{aligned}
&||\Sigma^{\frac{1-\gamma}{2}}q_\lambda(\hat \Sigma_{Id,\mathcal{A}_1})\mathbb{E}_P\xi(x,y)||=\lambda||\Sigma^{\frac{1-\gamma}{2}}q_\lambda(\hat \Sigma_{Id,\mathcal{A}_1})(\Sigma_{Id,\mathcal{A}_1}^\lambda)^{-1}\mathcal{A}_2S^\ast f_\rho||\\
&=\lambda||\Sigma^{\frac{1-\gamma}{2}}\underbrace{( \Sigma_{Id,\mathcal{A}_1}^\lambda)^{-1/2}( \Sigma_{Id,\mathcal{A}_1}^\lambda)^{1/2} }_{Id}\underbrace{(\hat \Sigma_{Id,\mathcal{A}_1}^\lambda)^{-1/2}(\hat \Sigma_{Id,\mathcal{A}_1}^\lambda)^{1/2} }_{Id}q(\hat \Sigma_{Id,\mathcal{A}_1}^\lambda) (\Sigma_{Id,\mathcal{A}_1}^\lambda) ^{1/2} (\Sigma_{Id,\mathcal{A}_1}^\lambda) ^{-3/2}\mathcal{A}_2S^\ast f_\rho||_{\mathcal{H}}\\
&\lesssim \lambda ||\Sigma^{\frac{1-\gamma}{2}}( \Sigma_{Id,\mathcal{A}_1}^\lambda)^{-1/2}|| ||(\Sigma_{Id,\mathcal{A}_1}^\lambda) ^{-1}||||(\Sigma_{Id,\mathcal{A}_1}^\lambda) ^{1/2}||||(\Sigma_{Id,\mathcal{A}_1}^\lambda) ^{-1}\mathcal{A}_2S^\ast f_\rho||\\
&\lesssim \lambda \lambda^{-\frac{\gamma\alpha+p}{2(\alpha+p)}} \lambda^{-1} \lambda^{1/2} \lambda^{\frac{\beta}{2(\alpha+p)}-\frac{1}{2}} \le \lambda^{-\frac{(\beta-\gamma)\alpha}{2(\alpha+p)}}
\end{aligned}
\end{equation*}

The last line is because of  $||\Sigma^{\frac{1-\gamma}{2}}( \Sigma_{Id,\mathcal{A}_1}^\lambda)^{-1/2}||^2\le\lambda^{-\frac{\gamma\alpha+p}{\alpha+p}}$ (Lemma \ref{lemma:soblevnormsigma}), $||( \Sigma_{Id,\mathcal{A}_1}^\lambda)^{1/2}(\hat \Sigma_{Id,\mathcal{A}_1}^\lambda)^{-1/2}||\le 2$ with high probability (Lemma \ref{lemma:aba}), $||(\Sigma_{Id,\mathcal{A}_1}^\lambda) ^{-1}||\le\lambda^{-1}$, $||(\Sigma_{Id,\mathcal{A}_1}^\lambda) ^{-1/2}\mathcal{A}_2S^\ast f_\rho||\le \lambda^{\frac{\beta}{2(\alpha+p)}-\frac{1}{2}}$ (proved while bounding term (\uppercase\expandafter{\romannumeral3})) and $p<0$.

\subsection{Final Bound} At this time we can have our final bound in Theorem \ref{theorem:earlystopping} via combining the bound for bias (Appendix \ref{section:appedixbias}) and  (Appendix \ref{section:appedixvaraince})

\begin{equation}
    \begin{aligned}
    ||\hat q_\lambda-u^*||_\gamma^2&\lesssim ||\hat q_\lambda-g_\lambda||_\gamma^2+||\hat q_\lambda-u^\ast||_\gamma^2\\
    &\lesssim \lambda^{\frac{(\beta-\gamma)\alpha}{\alpha+p}}+\frac{\tau^2}{n}\lambda^{-\frac{\gamma\alpha+p}{\alpha+p}}\left(\sigma^2 (\lambda)^{-\frac{1}{p+\alpha}-\frac{p-2q}{p+\alpha}}+\lambda^{-\frac{\mu\alpha-p}{\alpha+p}}\lambda^{\frac{\alpha\beta-2p}{\alpha+p}}+\frac{L_\lambda||h_x^2||+||h_x^1||}{n}\right)\\
    &\lesssim \lambda^{\frac{(\beta-\gamma)\alpha}{\alpha+p}} + \frac{\lambda^{-\frac{\gamma\alpha+2(p-q)+1}{p+\alpha}}}{n} +\frac{\lambda^{\frac{(\beta-\gamma)\alpha-\mu\alpha-p}{\alpha+p}}}{n}+\frac{\lambda^{-\frac{\mu\alpha+p+2q}{\alpha+p}}\lambda^{-\frac{\mu\alpha+p+2q}{\alpha+p}}}{n^2}
    \end{aligned}
\end{equation}

\paragraph{Case 1. $\beta\le \frac{\mu\alpha+2q-p+1}{\alpha}$} In this situation, $\frac{\lambda^{-\frac{\gamma\alpha+2(p-q)+1}{p+\alpha}}}{n} $is larger than $\frac{\lambda^{\frac{(\beta-\gamma)\alpha-\mu\alpha-p}{\alpha+p}}}{n}$. Thus  $\lambda^{\frac{(\beta-\gamma)\alpha}{\alpha+p}} + \frac{\lambda^{-\frac{\gamma\alpha+2(p-q)+1}{p+\alpha}}}{n}$ is the dominating term of the loss upper bound. Thus we can take $\lambda = n^{-\frac{\alpha+p}{\beta\alpha+2(p-q)+1}}$ and leads to $n^{-\frac{(\beta-\gamma)\alpha}{\beta+2(p-q)+1}}$ upper bound. At the same time, the third term is dominated by the second term.

\paragraph{Case 2. $\beta> \frac{\mu\alpha+2q-p+1}{\alpha}$}  In this situation, $\frac{\lambda^{\frac{(\beta-\gamma)\alpha-\mu\alpha-p}{\alpha+p}}}{n}$ is larger than $\frac{\lambda^{-\frac{\gamma\alpha+2(p-q)+1}{p+\alpha}}}{n} $. Thus  $\lambda^{\frac{(\beta-\gamma)\alpha}{\alpha+p}} + \frac{\lambda^{\frac{(\beta-\gamma)\alpha-\mu\alpha-p}{\alpha+p}}}{n}$ is the dominating term of the loss upper bound. Thus we can take $\lambda = n^{-\frac{\alpha+p}{\mu\alpha+p}}$ and leads to $n^{-\frac{(\beta-\gamma)\alpha}{\mu\alpha+p}}$ upper bound. At the same time, the third term is also dominated by the second term.

\section{Proof of the Lower Bound}
\label{appendix:low}
\subsection{Preliminaries on Tools for Lower Bounds}
\label{appendix:fano}
In this section, we repeat the standard tools we use to establish the lower bound. The main tool we use is the Fano's inequality and the Varshamov-Gilber Lemma.

\begin{lemma}[Fano's methods] Assume that $V$ is a uniform random variable over set $\mathcal{V}$, then for any Markov chain $V\rightarrow X\rightarrow \hat V$, we always have
$$
\mathcal{P}(\hat V\not = V)\ge 1-\frac{I(V;X)+\log 2}{\log(|\mathcal{V}|)}
$$

\end{lemma}

\begin{lemma}[Varshamov-Gillbert Lemma,\citep{tsybakov2008introduction} Theorem 2.9] Let $D\ge 8$. There exists a subset $\mathcal{V}=\{\tau^{(0)},\cdots,\tau^{(2^{D/8})}\}$ of $D-$dimensional hypercube $\mathcal{H}^D=\{0,1\}^D$ such that $\tau^{(0)}=(0,0,\cdots,0)$ and the $\ell_1$ distance between every two elements is larger than $\frac{D}{8}$
$$
\sum_{l=1}^D ||\tau^{(j)}-\tau^{(k)}||_{\ell_1}\ge \frac{D}{8}\text{, for all }0\le j,k \le 2^{D/8}
$$

\end{lemma}

\subsection{Proof of the Lower Bound}
\begin{theorem}
Let $(X,B)$ be a measurable space,  $H$ be a separable RKHS on $X$ respect to a bounded and measurable kernel $k$ and operator $\mathcal{A}=(\mathcal{A}_2^{-1}\mathcal{A}_1)$ satisfies Assumption \ref{assumption:kernel}. We have $n$ random observations $\{(x_i,y_i)\in\mathcal{X}\times\mathcal{Y}\}_{i=1}^n$ of $f^\ast = \mathcal{A}u, u\in\mathcal{H}^\gamma\cap L_\infty$, \emph{i.e.} $y_i=f^*(x_i)+\eta_i$ where $\eta_i$ is a random noise satisfies the momentum assumption $\mathbb{E}|\eta|^m\le\frac{1}{2}m!\sigma^2L^{m-2}$ for some constant $\sigma,L>0$. Then for all estimators $H:(\mathcal{X}\times\mathcal{Y})^{\otimes n}\rightarrow \mathcal{H}^\gamma$ satisfies
$$
\inf_{H}\sup_{u^\ast\in\mathcal{H}^\beta\cap L_\infty}\mathbb{E}||H(\{(x_i,y_i)\}_{i=1}^n)-u^\ast||_\gamma^2\gtrsim n^{-\frac{(\max\{\beta,\mu\}-\gamma)\alpha}{\max\{\beta,\mu\}\alpha+2(q-p)+1}}
$$
\end{theorem}

\begin{proof} To proof the lower bound, we use the standard Fano methods via reducing the lower bound to multiple hypothesis testing. We construct our hypothesis using binary strings $\omega=(\omega_1,\cdots,\omega_m)\in\{0,1\}^m$ ($m$ to be determined later) by defining
$$
u_\omega = \left(\frac{\epsilon}{m}\right)^{1/2}\sum_{i=1}^m \omega_i \mu_{i+m}^{\gamma/2} e_{i+m}.
$$

If we control \alertinline{$m\lesssim\epsilon^{-\frac{1}{\alpha\beta-\alpha\gamma}}$}, then we can  always keep $u_\omega\in\mathcal{H}^\beta$ for
$
||u_\omega||_\beta^2=\frac{\epsilon}{m}\sum_{i=1}^m \omega_i^2 \mu_{i+m}^{-(\beta-\gamma)}\lesssim \epsilon \mu_{2m}^{-(\beta-\gamma)}\lesssim m^{\alpha(\beta-\gamma)}\epsilon
=O(1)$. Similarly, we can select \alertinline{$m\lesssim\epsilon^{-\frac{1}{\alpha\mu-\alpha\gamma}}$} to control $||u_\omega||_{L^\infty}\le||u_\omega||_\mu\le O(1)$. At the same time, the associated PDE right hand side function$
f_\omega = \mathcal{A}_2^{-1}\mathcal{A}_1u_\omega = \left(\frac{\epsilon}{m}\right)^{1/2}\sum_{i=1}^m \frac{q_i\omega_i}{p_i} \mu_{i+m}^{\gamma/2} e_{i+m}
$.

Using Gilbert-Varshamov Lemma we know that there exists $M\ge 2^{m/8}$ binary strings $\omega^{(1)},\cdots,\omega^{(k)}\in\{0,1\}^m$ with $\omega^{(0)}=(0,\cdots,0)$ subject to
$$
\sum_{i=1}^m\left(\omega_i^{(j)}-\omega_i^{(k)}\right)^2\ge m/8
$$
holds for all $j\not =k$. As consequence, the distance between $f_\omega$ and $f_{\omega'}$ can be lower bounded as
$||u_\omega-u_{\omega'}||_\gamma^2=\frac{\epsilon}{m}\sum_{i=1}^m(\omega_i-\omega_i')^2\ge \epsilon/8$. To apply the Fano method, we still need to bound the mutual information between the uniform distribution over all the hypothesis and the distribution of the observed data. We take $\eta_i$ is sampled form $\mathcal{N}(0,\min\{\sigma,L\}^2)$ which satisfies the momentum condition. Then we know that this mutual information can be bounded by the following average of KL divergence\citep{tsybakov2008introduction} via

\begin{equation}
    \begin{aligned}
       I(V,X)=\frac{1}{M_\epsilon} \sum_{j=1}^{M_\epsilon} KL(P_j^n||P_0^n)=\frac{n}{2\bar\sigma^2 M_\epsilon}\sum_{j=1}^{M_\epsilon}||f_j-f_0||^2_{L_2}\lesssim {n\epsilon m_\epsilon^{-\alpha\gamma+2(p-q)}}
    \end{aligned}
\end{equation}

Then we apply the Fano's inequality

\begin{equation*}
    \begin{aligned}
       \mathbb{P}(\hat V\not= V)&\ge 1-\frac{I(V;X)+\log 2}{\log |V|}=1-\frac{\frac{16C^\gamma}{\min\{\sigma,L\}^2}{n\epsilon m_\epsilon^{-\alpha\gamma-2(p-q)}}+\log 2}{\frac{\log 2}{8}m_\epsilon}\\
       &=1-O(n\epsilon\epsilon^{\frac{1+\alpha\gamma+2(p-q)}{\alpha(\max\{\beta,\mu\}-\gamma)}})
    \end{aligned}
\end{equation*}

Take $\epsilon = n^{-\frac{(\max\{\beta,\mu\}-\gamma)\alpha}{\max\{\beta,\mu\}\alpha+2(p-q)+1}}$, we know that with constant probability we have
$$
||H(\{(x_i,y_i)\}_{i=1}^n)-u^\ast||_\gamma^2\gtrsim n^{-\frac{(\max\{\beta,\mu\}-\gamma)\alpha}{\max\{\beta,\mu\}\alpha+2(p-q)+1}}
$$
\end{proof}

%% file: example.bib
@article{duan2021convergence,
  title={Convergence Rate Analysis for Deep Ritz Method},
  author={Duan, Chenguang and Jiao, Yuling and Lai, Yanming and Lu, Xiliang and Yang, Zhijian},
  journal={arXiv preprint arXiv:2103.13330},
  year={2021}
}

@article{stepaniants2021learning,
  title={Learning partial differential equations in reproducing kernel hilbert spaces},
  author={Stepaniants, George},
  journal={arXiv preprint arXiv:2108.11580},
  year={2021}
}

@article{lu2021priori,
  title={A Priori Generalization Analysis of the Deep Ritz Method for Solving High Dimensional Elliptic Equations},
  author={Lu, Jianfeng and Lu, Yulong and Wang, Min},
  journal={arXiv preprint arXiv:2101.01708},
  year={2021}
}

@article{chen2021solving,
  title={Solving and learning nonlinear PDEs with gaussian processes},
  author={Chen, Yifan and Hosseini, Bamdad and Owhadi, Houman and Stuart, Andrew M},
  journal={arXiv preprint arXiv:2103.12959},
  year={2021}
}

@article{weinan2018deep,
  title={The deep Ritz method: a deep learning-based numerical algorithm for solving variational problems},
  author={E, Weinan and Yu, Bing},
  journal={Communications in Mathematics and Statistics},
  volume={6},
  number={1},
  pages={1--12},
  year={2018},
  publisher={Springer}
}

@article{khoo2017solving,
  title={Solving parametric PDE problems with artificial neural networks},
  author={Khoo, Yuehaw and Lu, Jianfeng and Ying, Lexing},
  journal={arXiv preprint arXiv:1707.03351},
  year={2017}
}

@article{marwah2021parametric,
  title={Parametric Complexity Bounds for Approximating PDEs with Neural Networks},
  author={Marwah, Tanya and Lipton, Zachary C and Risteski, Andrej},
  journal={arXiv preprint arXiv:2103.02138},
  year={2021}
}

@article{xu2020finite,
  title={The finite neuron method and convergence analysis},
  author={Xu, Jinchao},
  journal={arXiv preprint arXiv:2010.01458},
  year={2020}
}

@article{shin2020error,
  title={Error estimates of residual minimization using neural networks for linear PDEs},
  author={Shin, Yeonjong and Zhang, Zhongqiang and Karniadakis, George Em},
  journal={arXiv preprint arXiv:2010.08019},
  year={2020}
}

@article{luo2020two,
  title={Two-layer neural networks for partial differential equations: Optimization and generalization theory},
  author={Luo, Tao and Yang, Haizhao},
  journal={arXiv preprint arXiv:2006.15733},
  year={2020}
}

@article{grohs2020deep,
  title={Deep neural network approximation for high-dimensional elliptic PDEs with boundary conditions},
  author={Grohs, Philipp and Herrmann, Lukas},
  journal={arXiv preprint arXiv:2007.05384},
  year={2020}
}

@article{wojtowytsch2020some,
  title={Some observations on partial differential equations in Barron and multi-layer spaces},
  author={Wojtowytsch, Stephan and others},
  journal={arXiv preprint arXiv:2012.01484},
  year={2020}
}

@article{raissi2019physics,
  title={Physics-informed neural networks: A deep learning framework for solving forward and inverse problems involving nonlinear partial differential equations},
  author={Raissi, Maziar and Perdikaris, Paris and Karniadakis, George E},
  journal={Journal of Computational Physics},
  volume={378},
  pages={686--707},
  year={2019},
  publisher={Elsevier}
}

@article{sirignano2018dgm,
  title={DGM: A deep learning algorithm for solving partial differential equations},
  author={Sirignano, Justin and Spiliopoulos, Konstantinos},
  journal={Journal of computational physics},
  volume={375},
  pages={1339--1364},
  year={2018},
  publisher={Elsevier}
}

@article{han2018solving,
  title={Solving high-dimensional partial differential equations using deep learning},
  author={Han, Jiequn and Jentzen, Arnulf and Weinan, E},
  journal={Proceedings of the National Academy of Sciences},
  volume={115},
  number={34},
  pages={8505--8510},
  year={2018},
  publisher={National Acad Sciences}
}

@inproceedings{long2018pde,
  title={Pde-net: Learning pdes from data},
  author={Long, Zichao and Lu, Yiping and Ma, Xianzhong and Dong, Bin},
  booktitle={International Conference on Machine Learning},
  pages={3208--3216},
  year={2018},
  organization={PMLR}
}

@article{zang2020weak,
  title={Weak adversarial networks for high-dimensional partial differential equations},
  author={Zang, Yaohua and Bao, Gang and Ye, Xiaojing and Zhou, Haomin},
  journal={Journal of Computational Physics},
  volume={411},
  pages={109409},
  year={2020},
  publisher={Elsevier}
}

@article{manole2021plugin,
  title={Plugin Estimation of Smooth Optimal Transport Maps},
  author={Manole, Tudor and Balakrishnan, Sivaraman and Niles-Weed, Jonathan and Wasserman, Larry},
  journal={arXiv preprint arXiv:2107.12364},
  year={2021}
}

@article{chen2020comprehensive,
  title={A comprehensive study of boundary conditions when solving PDEs by DNNs},
  author={Chen, Jingrun and Du, Rui and Wu, Keke},
  journal={arXiv preprint arXiv:2005.04554},
  year={2020}
}

@article{long2019pde,
  title={PDE-Net 2.0: Learning PDEs from data with a numeric-symbolic hybrid deep network},
  author={Long, Zichao and Lu, Yiping and Dong, Bin},
  journal={Journal of Computational Physics},
  volume={399},
  pages={108925},
  year={2019},
  publisher={Elsevier}
}

@article{nitanda2020optimal,
  title={Optimal Rates for Averaged Stochastic Gradient Descent under Neural Tangent Kernel Regime},
  author={Nitanda, Atsushi and Suzuki, Taiji},
  journal={arXiv preprint arXiv:2006.12297},
  year={2020}
}

@article{hutter2019minimax,
  title={Minimax rates of estimation for smooth optimal transport maps},
  author={H{\"u}tter, Jan-Christian and Rigollet, Philippe},
  journal={arXiv preprint arXiv:1905.05828},
  year={2019}
}

@book{tsybakov2008introduction,
  title={Introduction to nonparametric estimation},
  author={Tsybakov, Alexandre B},
  year={2008},
  publisher={Springer Science \& Business Media}
}

@article{bai2021physics,
  title={Physics Informed Neural Networks (PINNs) for approximating nonlinear dispersive PDEs},
  author={Bai, Genming and Koley, Ujjwal and Mishra, Siddhartha and Molinaro, Roberto},
  journal={arXiv preprint arXiv:2104.05584},
  year={2021}
}

@article{nickl2020convergence,
  title={Convergence Rates for Penalized Least Squares Estimators in PDE Constrained Regression Problems},
  author={Nickl, Richard and van de Geer, Sara and Wang, Sven},
  journal={SIAM/ASA Journal on Uncertainty Quantification},
  volume={8},
  number={1},
  pages={374--413},
  year={2020},
  publisher={SIAM}
}

@misc{jiao2021convergence,
      title={Convergence Analysis for the PINNs}, 
      author={Yuling Jiao and Yanming Lai and Dingwei Li and Xiliang Lu and Yang Wang and Jerry Zhijian Yang},
      year={2021},
      eprint={2109.01780},
      archivePrefix={arXiv},
      primaryClass={math.NA}
}

@article{jiao2021error,
  title={Error Analysis of Deep Ritz Methods for Elliptic Equations},
  author={Jiao, Yuling and Lai, Yanming and Luo, Yisu and Wang, Yang and Yang, Yunfei},
  journal={arXiv preprint arXiv:2107.14478},
  year={2021}
}

@article{han2020solvingeigen,
  title={Solving high-dimensional eigenvalue problems using deep neural networks: A diffusion Monte Carlo like approach},
  author={Han, Jiequn and Lu, Jianfeng and Zhou, Mo},
  journal={Journal of Computational Physics},
  volume={423},
  pages={109792},
  year={2020},
  publisher={Elsevier}
}

@book{adams2003sobolev,
  title={Sobolev spaces},
  author={Adams, Robert A and Fournier, John JF},
  year={2003},
  publisher={Elsevier}
}

@article{hu2020regularization, title={Regularization Matters: A Nonparametric Perspective on Overparametrized Neural Network}, author={Hu, Tianyang and Wang, Wenjia and Lin, Cong and Cheng, Guang}, journal={arXiv preprint arXiv:2007.02486}, year={2020} }

@article{tsybakov2004optimal,
  title={Optimal aggregation of classifiers in statistical learning},
  author={Tsybakov, Alexander B},
  journal={The Annals of Statistics},
  volume={32},
  number={1},
  pages={135--166},
  year={2004},
  publisher={Institute of Mathematical Statistics}
}

@article{pillaud2018statistical,
  title={Statistical optimality of stochastic gradient descent on hard learning problems through multiple passes},
  author={Pillaud-Vivien, Loucas and Rudi, Alessandro and Bach, Francis},
  journal={arXiv preprint arXiv:1805.10074},
  year={2018}
}

@article{dieuleveut2016nonparametric,
  title={Nonparametric stochastic approximation with large step-sizes},
  author={Dieuleveut, Aymeric and Bach, Francis},
  journal={The Annals of Statistics},
  volume={44},
  number={4},
  pages={1363--1399},
  year={2016},
  publisher={Institute of Mathematical Statistics}
}

@article{lin2017optimal,
  title={Optimal rates for multi-pass stochastic gradient methods},
  author={Lin, Junhong and Rosasco, Lorenzo},
  journal={The Journal of Machine Learning Research},
  volume={18},
  number={1},
  pages={3375--3421},
  year={2017},
  publisher={JMLR. org}
}

@article{polyak1992acceleration,
  title={Acceleration of stochastic approximation by averaging},
  author={Polyak, Boris T and Juditsky, Anatoli B},
  journal={SIAM journal on control and optimization},
  volume={30},
  number={4},
  pages={838--855},
  year={1992},
  publisher={SIAM}
}

@article{yao2007early,
  title={On early stopping in gradient descent learning},
  author={Yao, Yuan and Rosasco, Lorenzo and Caponnetto, Andrea},
  journal={Constructive Approximation},
  volume={26},
  number={2},
  pages={289--315},
  year={2007},
  publisher={Springer}
}

@article{lu2021machine,
  title={Machine Learning For Elliptic PDEs: Fast Rate Generalization Bound, Neural Scaling Law and Minimax Optimality},
  author={Lu, Yiping and Chen, Haoxuan and Lu, Jianfeng and Ying, Lexing and Blanchet, Jose},
  journal={arXiv preprint arXiv:2110.06897},
  year={2021}
}

@article{fischer2020sobolev,
  title={Sobolev Norm Learning Rates for Regularized Least-Squares Algorithms.},
  author={Fischer, Simon and Steinwart, Ingo},
  journal={J. Mach. Learn. Res.},
  volume={21},
  pages={205--1},
  year={2020}
}

@article{chen2020deep,
  title={Deep neural tangent kernel and laplace kernel have the same RKHS},
  author={Chen, Lin and Xu, Sheng},
  journal={arXiv preprint arXiv:2009.10683},
  year={2020}
}

@article{bietti2020deep,
  title={Deep equals shallow for ReLU networks in kernel regimes},
  author={Bietti, Alberto and Bach, Francis},
  journal={arXiv preprint arXiv:2009.14397},
  year={2020}
}

@article{caponnetto2007optimal,
  title={Optimal rates for the regularized least-squares algorithm},
  author={Caponnetto, Andrea and De Vito, Ernesto},
  journal={Foundations of Computational Mathematics},
  volume={7},
  number={3},
  pages={331--368},
  year={2007},
  publisher={Springer}
}

@inproceedings{steinwart2009optimal,
  title={Optimal Rates for Regularized Least Squares Regression.},
  author={Steinwart, Ingo and Hush, Don R and Scovel, Clint and others},
  booktitle={COLT},
  pages={79--93},
  year={2009}
}

@article{lei2021generalization,
  title={Generalization Performance of Multi-pass Stochastic Gradient Descent with Convex Loss Functions.},
  author={Lei, Yunwen and Hu, Ting and Tang, Ke},
  journal={J. Mach. Learn. Res.},
  volume={22},
  pages={25--1},
  year={2021}
}

@inproceedings{cabannes2021overcoming,
  title={Overcoming the curse of dimensionality with Laplacian regularization in semi-supervised learning},
  author={Cabannes, Vivien and Pillaud-Vivien, Loucas and Bach, Francis and Rudi, Alessandro},
  booktitle={Thirty-Fifth Conference on Neural Information Processing Systems},
  year={2021}
}

@article{yang2021implicit,
  title={Implicit Regularization Effects of the Sobolev Norms in Image Processing},
  author={Yang, Yunan and Hu, Jingwei and Lou, Yifei},
  journal={arXiv preprint arXiv:2109.06255},
  year={2021}
}

@article{calder2010image,
  title={Image sharpening via Sobolev gradient flows},
  author={Calder, Jeff and Mansouri, A and Yezzi, Anthony},
  journal={SIAM Journal on Imaging Sciences},
  volume={3},
  number={4},
  pages={981--1014},
  year={2010},
  publisher={SIAM}
}

@article{richardson2008sobolev,
  title={Sobolev gradient preconditioning for image-processing PDEs},
  author={Richardson Jr, WB},
  journal={Communications in Numerical Methods in Engineering},
  volume={24},
  number={6},
  pages={493--504},
  year={2008},
  publisher={Wiley Online Library}
}

@article{son2021sobolev,
  title={Sobolev Training for the Neural Network Solutions of PDEs},
  author={Son, Hwijae and Jang, Jin Woo and Han, Woo Jin and Hwang, Hyung Ju},
  journal={arXiv preprint arXiv:2101.08932},
  year={2021}
}

@article{czarnecki2017sobolev,
  title={Sobolev training for neural networks},
  author={Czarnecki, Wojciech Marian and Osindero, Simon and Jaderberg, Max and {\'S}wirszcz, Grzegorz and Pascanu, Razvan},
  journal={arXiv preprint arXiv:1706.04859},
  year={2017}
}

@article{yu2021repulsive,
  title={Repulsive Curves},
  author={Yu, Chris and Schumacher, Henrik and Crane, Keenan},
  journal={ACM Transactions on Graphics (TOG)},
  volume={40},
  number={2},
  pages={1--21},
  year={2021},
  publisher={ACM New York, NY}
}

@article{yu2021repulsiveb,
  title={Repulsive Surfaces},
  author={Yu, Chris and Brakensiek, Caleb and Schumacher, Henrik and Crane, Keenan},
  journal={arXiv preprint arXiv:2107.01664},
  year={2021}
}

@article{soliman2021constrained,
  title={Constrained willmore surfaces},
  author={Soliman, Yousuf and Chern, Albert and Diamanti, Olga and Kn{\"o}ppel, Felix and Pinkall, Ulrich and Schr{\"o}der, Peter},
  journal={ACM Transactions on Graphics (TOG)},
  volume={40},
  number={4},
  pages={1--17},
  year={2021},
  publisher={ACM New York, NY, USA}
}

@inproceedings{rahaman2019spectral,
  title={On the spectral bias of neural networks},
  author={Rahaman, Nasim and Baratin, Aristide and Arpit, Devansh and Draxler, Felix and Lin, Min and Hamprecht, Fred and Bengio, Yoshua and Courville, Aaron},
  booktitle={International Conference on Machine Learning},
  pages={5301--5310},
  year={2019},
  organization={PMLR}
}

@article{xu2019frequency,
  title={Frequency principle: Fourier analysis sheds light on deep neural networks},
  author={Xu, Zhi-Qin John and Zhang, Yaoyu and Luo, Tao and Xiao, Yanyang and Ma, Zheng},
  journal={arXiv preprint arXiv:1901.06523},
  year={2019}
}

@article{kalimeris2019sgd,
  title={Sgd on neural networks learns functions of increasing complexity},
  author={Kalimeris, Dimitris and Kaplun, Gal and Nakkiran, Preetum and Edelman, Benjamin and Yang, Tristan and Barak, Boaz and Zhang, Haofeng},
  journal={Advances in Neural Information Processing Systems},
  volume={32},
  pages={3496--3506},
  year={2019}
}

@inproceedings{murata2021gradient,
  title={Gradient Descent in RKHS with Importance Labeling},
  author={Murata, Tomoya and Suzuki, Taiji},
  booktitle={International Conference on Artificial Intelligence and Statistics},
  pages={1981--1989},
  year={2021},
  organization={PMLR}
}

@article{smale2007learning,
  title={Learning theory estimates via integral operators and their approximations},
  author={Smale, Steve and Zhou, Ding-Xuan},
  journal={Constructive approximation},
  volume={26},
  number={2},
  pages={153--172},
  year={2007},
  publisher={Springer}
}

@article{lin2020optimal,
  title={Optimal rates for spectral algorithms with least-squares regression over hilbert spaces},
  author={Lin, Junhong and Rudi, Alessandro and Rosasco, Lorenzo and Cevher, Volkan},
  journal={Applied and Computational Harmonic Analysis},
  volume={48},
  number={3},
  pages={868--890},
  year={2020},
  publisher={Elsevier}
}

@article{de2005learning,
  title={Learning from Examples as an Inverse Problem.},
  author={De Vito, Ernesto and Rosasco, Lorenzo and Caponnetto, Andrea and De Giovannini, Umberto and Odone, Francesca and Bartlett, Peter},
  journal={Journal of Machine Learning Research},
  volume={6},
  number={5},
  year={2005}
}

@article{rosasco2010learning,
  title={On learning with integral operators.},
  author={Rosasco, Lorenzo and Belkin, Mikhail and De Vito, Ernesto},
  journal={Journal of Machine Learning Research},
  volume={11},
  number={2},
  year={2010}
}

@article{liu2020estimation,
  title={On the Estimation of Derivatives Using Plug-in KRR Estimators},
  author={Liu, Zejian and Li, Meng},
  journal={arXiv preprint arXiv:2006.01350},
  year={2020}
}

@article{wei2017early,
  title={Early stopping for kernel boosting algorithms: A general analysis with localized complexities},
  author={Wei, Yuting and Yang, Fanny and Wainwright, Martin J},
  journal={arXiv preprint arXiv:1707.01543},
  year={2017}
}

@article{jacot2018neural,
  title={Neural tangent kernel: Convergence and generalization in neural networks},
  author={Jacot, Arthur and Gabriel, Franck and Hongler, Cl{\'e}ment},
  journal={arXiv preprint arXiv:1806.07572},
  year={2018}
}

@article{daniely2017sgd,
  title={SGD learns the conjugate kernel class of the network},
  author={Daniely, Amit},
  journal={arXiv preprint arXiv:1702.08503},
  year={2017}
}

@article{lee2017deep,
  title={Deep neural networks as gaussian processes},
  author={Lee, Jaehoon and Bahri, Yasaman and Novak, Roman and Schoenholz, Samuel S and Pennington, Jeffrey and Sohl-Dickstein, Jascha},
  journal={arXiv preprint arXiv:1711.00165},
  year={2017}
}

@article{stuart2010inverse,
  title={Inverse problems: a Bayesian perspective},
  author={Stuart, Andrew M},
  journal={Acta numerica},
  volume={19},
  pages={451--559},
  year={2010},
  publisher={Cambridge University Press}
}

@article{benning2018modern,
  title={Modern regularization methods for inverse problems},
  author={Benning, Martin and Burger, Martin},
  journal={Acta Numerica},
  volume={27},
  pages={1--111},
  year={2018},
  publisher={Cambridge University Press}
}

@article{amari2020does,
  title={When Does Preconditioning Help or Hurt Generalization?},
  author={Amari, Shun-ichi and Ba, Jimmy and Grosse, Roger and Li, Xuechen and Nitanda, Atsushi and Suzuki, Taiji and Wu, Denny and Xu, Ji},
  journal={arXiv preprint arXiv:2006.10732},
  year={2020}
}

@article{vavskevivcius2020statistical,
  title={The statistical complexity of early-stopped mirror descent},
  author={Va{\v{s}}kevi{\v{c}}ius, Tomas and Kanade, Varun and Rebeschini, Patrick},
  journal={arXiv preprint arXiv:2002.00189},
  year={2020}
}

@article{pagliana2019implicit,
  title={Implicit regularization of accelerated methods in hilbert spaces},
  author={Pagliana, Nicol{\`o} and Rosasco, Lorenzo},
  journal={arXiv preprint arXiv:1905.13000},
  year={2019}
}

@article{steinwart2012mercer,
  title={Mercer’s theorem on general domains: On the interaction between measures, kernels, and RKHSs},
  author={Steinwart, Ingo and Scovel, Clint},
  journal={Constructive Approximation},
  volume={35},
  number={3},
  pages={363--417},
  year={2012},
  publisher={Springer}
}

@article{blanchard2018optimal,
  title={Optimal rates for regularization of statistical inverse learning problems},
  author={Blanchard, Gilles and M{\"u}cke, Nicole},
  journal={Foundations of Computational Mathematics},
  volume={18},
  number={4},
  pages={971--1013},
  year={2018},
  publisher={Springer}
}

@article{gerfo2008spectral,
  title={Spectral algorithms for supervised learning},
  author={Gerfo, L Lo and Rosasco, Lorenzo and Odone, Francesca and Vito, E De and Verri, Alessandro},
  journal={Neural Computation},
  volume={20},
  number={7},
  pages={1873--1897},
  year={2008},
  publisher={MIT Press One Rogers Street, Cambridge, MA 02142-1209, USA journals-info~…}
}

@article{colbrook2021computing,
  title={Computing spectral measures of self-adjoint operators},
  author={Colbrook, Matthew and Horning, Andrew and Townsend, Alex},
  journal={SIAM Review},
  volume={63},
  number={3},
  pages={489--524},
  year={2021},
  publisher={SIAM}
}

@inproceedings{marteau2019beyond,
  title={Beyond least-squares: Fast rates for regularized empirical risk minimization through self-concordance},
  author={Marteau-Ferey, Ulysse and Ostrovskii, Dmitrii and Bach, Francis and Rudi, Alessandro},
  booktitle={Conference on Learning Theory},
  pages={2294--2340},
  year={2019},
  organization={PMLR}
}

@book{stein1999interpolation,
  title={Interpolation of spatial data: some theory for kriging},
  author={Stein, Michael L},
  year={1999},
  publisher={Springer Science \& Business Media}
}

@article{yu2021gradient,
  title={Gradient-enhanced physics-informed neural networks for forward and inverse PDE problems},
  author={Yu, Jeremy and Lu, Lu and Meng, Xuhui and Karniadakis, George Em},
  journal={arXiv preprint arXiv:2111.02801},
  year={2021}
}

@article{de2021convergence,
  title={Convergence rates for learning linear operators from noisy data},
  author={de Hoop, Maarten V and Kovachki, Nikola B and Nelsen, Nicholas H and Stuart, Andrew M},
  journal={arXiv preprint arXiv:2108.12515},
  year={2021}
}

@inproceedings{zhou2011semi,
  title={Semi-supervised learning by higher order regularization},
  author={Zhou, Xueyuan and Belkin, Mikhail},
  booktitle={Proceedings of the fourteenth international conference on artificial intelligence and statistics},
  pages={892--900},
  year={2011},
  organization={JMLR Workshop and Conference Proceedings}
}

@inproceedings{scetbon2021spectral,
  title={A Spectral Analysis of Dot-product Kernels},
  author={Scetbon, Meyer and Harchaoui, Zaid},
  booktitle={International Conference on Artificial Intelligence and Statistics},
  pages={3394--3402},
  year={2021},
  organization={PMLR}
}

@inproceedings{ronneberger2015u,
  title={U-net: Convolutional networks for biomedical image segmentation},
  author={Ronneberger, Olaf and Fischer, Philipp and Brox, Thomas},
  booktitle={International Conference on Medical image computing and computer-assisted intervention},
  pages={234--241},
  year={2015},
  organization={Springer}
}

@article{sitzmann2020implicit,
  title={Implicit neural representations with periodic activation functions},
  author={Sitzmann, Vincent and Martel, Julien NP and Bergman, Alexander W and Lindell, David B and Wetzstein, Gordon},
  journal={arXiv preprint arXiv:2006.09661},
  year={2020}
}

@article{nickl2020polynomial,
  title={On polynomial-time computation of high-dimensional posterior measures by Langevin-type algorithms},
  author={Nickl, Richard and Wang, Sven},
  journal={arXiv preprint arXiv:2009.05298},
  year={2020}
}

@inproceedings{ali2019continuous,
  title={A continuous-time view of early stopping for least squares regression},
  author={Ali, Alnur and Kolter, J Zico and Tibshirani, Ryan J},
  booktitle={The 22nd International Conference on Artificial Intelligence and Statistics},
  pages={1370--1378},
  year={2019},
  organization={PMLR}
}

@inproceedings{ali2020implicit,
  title={The implicit regularization of stochastic gradient flow for least squares},
  author={Ali, Alnur and Dobriban, Edgar and Tibshirani, Ryan},
  booktitle={International Conference on Machine Learning},
  pages={233--244},
  year={2020},
  organization={PMLR}
}

@article{abraham2019statistical,
  title={On statistical Calder$\backslash$'on problems},
  author={Abraham, Kweku and Nickl, Richard},
  journal={arXiv preprint arXiv:1906.03486},
  year={2019}
}

@article{monard2021consistent,
  title={Consistent Inversion of Noisy Non-Abelian X-Ray Transforms},
  author={Monard, Fran{\c{c}}ois and Nickl, Richard and Paternain, Gabriel P},
  journal={Communications on Pure and Applied Mathematics},
  volume={74},
  number={5},
  pages={1045--1099},
  year={2021},
  publisher={Wiley Online Library}
}

@article{lyu2020mim,
  title={MIM: A deep mixed residual method for solving high-order partial differential equations},
  author={Lyu, Liyao and Zhang, Zhen and Chen, Minxin and Chen, Jingrun},
  journal={arXiv preprint arXiv:2006.04146},
  year={2020}
}

@article{zhu2021local,
  title={A Local Deep Learning Method for Solving High Order Partial Differential Equations},
  author={Zhu, Quanhui and Yang, Jiang},
  journal={arXiv preprint arXiv:2103.08915},
  year={2021}
}

@article{cao2019generalization,
  title={Generalization bounds of stochastic gradient descent for wide and deep neural networks},
  author={Cao, Yuan and Gu, Quanquan},
  journal={Advances in Neural Information Processing Systems},
  volume={32},
  pages={10836--10846},
  year={2019}
}

@article{pinelis1985remarks,
  title={Remarks on inequalities for the probabilities of large deviations},
  author={Pinelis, Iosif F and Sakhanenko, Aleksandr Ivanovich},
  journal={Teoriya Veroyatnostei i ee Primeneniya},
  volume={30},
  number={1},
  pages={127--131},
  year={1985},
  publisher={Russian Academy of Sciences, Steklov Mathematical Institute of Russian~…}
}

@article{dicker2017kernel,
  title={Kernel ridge vs. principal component regression: Minimax bounds and the qualification of regularization operators},
  author={Dicker, Lee H and Foster, Dean P and Hsu, Daniel},
  journal={Electronic Journal of Statistics},
  volume={11},
  number={1},
  pages={1022--1047},
  year={2017},
  publisher={Institute of Mathematical Statistics and Bernoulli Society}
}

@article{mendelson2010regularization,
  title={Regularization in kernel learning},
  author={Mendelson, Shahar and Neeman, Joseph},
  journal={The Annals of Statistics},
  volume={38},
  number={1},
  pages={526--565},
  year={2010},
  publisher={Institute of Mathematical Statistics}
}

@article{mucke2020stochastic,
  title={Stochastic Gradient Descent in Hilbert Scales: Smoothness, Preconditioning and Earlier Stopping},
  author={M{\"u}cke, Nicole and Reiss, Enrico},
  journal={arXiv preprint arXiv:2006.10840},
  year={2020}
}

@article{fasshauer2011reproducing,
  title={Reproducing kernels of generalized Sobolev spaces via a Green function approach with distributional operators},
  author={Fasshauer, Gregory E and Ye, Qi},
  journal={Numerische Mathematik},
  volume={119},
  number={3},
  pages={585--611},
  year={2011},
  publisher={Springer}
}

@article{shi2010hermite,
  title={Hermite learning with gradient data},
  author={Shi, Lei and Guo, Xin and Zhou, Ding-Xuan},
  journal={Journal of computational and applied mathematics},
  volume={233},
  number={11},
  pages={3046--3059},
  year={2010},
  publisher={Elsevier}
}
